\documentclass[reqno,11pt]{amsart}
\usepackage{amsmath,amsthm,amssymb,amsfonts,epsfig,color,graphicx,enumerate,enumitem,accents,subfigure}
\usepackage[a4paper,margin=3truecm]{geometry}
\usepackage[T1]{fontenc}

\DeclareMathOperator{\argmin}{argmin}

\DeclareMathOperator{\spt}{spt}

\DeclareMathOperator{\dist}{dist}
\def\a{\alpha}

\def\ds{\displaystyle}
\def\eps{{\varepsilon}}
\def\N{\mathbb{N}}
\def\R{\mathbb{R}}
\def\CC{\mathbb{C}}
\def\O{\Omega}

\def\Cb{\overline{C}}
\def\rhob{\overline{\rho}}
\def\rhop{\rho^ \perp}

\def\E{\mathcal{E}}

\def\HH{\mathcal{H}}

\def\PP{\mathcal{P}}

\def\C{\mathcal{C}}

\def\Tcal{\mathcal{T}}
\def\Pcal{\mathcal{P}}
\def\Ccal{\mathcal{C}}
\def\Ucal{\mathcal{U}}
\newcommand{\be}{\begin{equation}}
\newcommand{\ee}{\end{equation}}
\newcommand{\bib}[4]{\bibitem{#1}{\sc#2: }{\it#3. }{#4.}}

\newcommand{\weak}{\stackrel{*}{\rightharpoonup}}

\newcommand{\norm}[1]{\left\lvert #1 \right\rvert}
\newcommand{\Norm}[1]{\left\lVert #1 \right\rVert}

\numberwithin{equation}{section}
\theoremstyle{plain}
\newtheorem{theo}{Theorem}[section]
\newtheorem{lemm}[theo]{Lemma}

\newtheorem{prop}[theo]{Proposition}
\newtheorem{defi}[theo]{Definition}
\theoremstyle{remark}
\newtheorem{rema}[theo]{\bf Remark}
\newtheorem{exam}[theo]{Example}

\title[Bond dissociating limit in DFT with strongly correlated electrons]{Dissociating limit in Density Functional Theory with Coulomb optimal transport cost}

\author{Guy Bouchitt\'e, Giuseppe Buttazzo, Thierry Champion, Luigi De Pascale}

\begin{document}

\begin{abstract} In the framework of Density Functional Theory with Strongly Correlated Electrons we consider the so called bond dissociating limit for the energy of an aggregate of atoms. We show that the multi-marginals optimal transport cost with Coulombian electron-electron repulsion may correctly describe the dissociation effect. The variational limit is completely calculated in the case of $N=2$ electrons. The theme of fractional number of electrons appears naturally and brings into play the question of optimal partial transport cost. A plan is outlined to complete the analysis which involves the study of the relaxation of optimal transport cost with respect to the weak* convergence of measures.
\end{abstract}

\maketitle



%%%%%%%%%%%%%%%%%%%%%%%%%%%%%%%%%%%%%%%%%%%%%%%%%%
\section{Introduction}%\label{sintro}

The analysis of minimum problems for functionals involving the wave function is one of the most studied topics in quantum physics. The Born-Oppenheimer model for the electronic structure of several particles systems deals with the functional
\be\label{wavef}
\E(\psi)=\frac{\hbar^2}{2m}T(\psi)+U_{ee}(\psi)-U_{ne}(\psi)
\ee
where
\[\begin{split}
&T(\psi)=\sum_{s_i=\pm1}\int_{\R^{3N}}\sum_{1\le i\le N}|\nabla_{x_i}\psi|^2\,dx_1\dots dx_N\quad\hbox{(kinetic energy)},\\
&U_{ee}(\psi)=\sum_{s_i=\pm1}\int_{\R^{3N}}\bigg(\sum_{1\le i<j\le N}\frac{1}{|x_i-x_j|}\bigg)|\psi|^2\,dx_1\dots dx_N\quad\hbox{(electron-electron interaction)},\\
&U_{ne}(\psi)=\sum_{s_i=\pm1}\int_{\R^{3N}}V(x_1,\dots,x_N)|\psi|^2\,dx_1\dots dx_N\quad\hbox{(electron-nuclei interaction)}.
\end{split}\]
We do not consider the nucleus-nucleus interaction because in our case we assume the nuclei are fixed and this extra term would then simply be a constant. However, if the nuclei are not considered as fixed, an extra term involving a repulsive nucleus-nucleus interaction has to be added.

Here $\psi(x_1,s_1,x_2,s_2,\dots,x_N,s_N)$ is the wave function depending on space coordinates $x_i$ and spin coordinates $s_i$, $m$ the reduced mass of the nuclei, $N$ the number of electrons, and $V$ a potential. In the Coulomb case, if we assume to have $M$ nuclei with positions $X_k$ and charges $Z_k$ ($k=1,\dots,M$), we may take
$$V(x_1,\dots,x_N)=\sum_{1\le k\le M}\frac{Z_k}{|x-X_k|}\;,$$
even if most of the analysis can be similarly carried out assuming only $V\in L^{3/2}(\R^3)+L^\infty(\R^3)$. The class where the functional above has to be minimized is the class of functions that, with respect to the space variables, belong to the Sobolev space $H^1(\R^{3N};\CC)$ with $\|\psi\|_{L^2}=1$.

Concerning the symmetry assumptions on the functions $\psi$ in the admissible class, there are two main cases considered in physics:
\begin{itemize}
\item the {\em bosonic} case, where for all permutations $\sigma$ of $N$ points
$$\psi\big((x_{\sigma(1)},s_{\sigma(1)}),\dots,(x_{\sigma(N)},s_{\sigma(N)})\big)=\psi\big((x_1,s_1),\dots,(x_N,s_N)\big)$$
\item the {\em fermionic} case, where for all permutations $\sigma$ of $N$ points 
$$\psi\big((x_{\sigma(1)},s_{\sigma(1)}),\dots,(x_{\sigma(N)},s_{\sigma(N)})\big)=sign(\sigma)\psi\big((x_1,s_1),\dots,(x_N,s_N)\big)$$
\end{itemize}

To simplify this rather complex problem, several approximated models have been proposed; here we deal with the one considered in the Density Functional Theory (DFT) introduced in the works of Thomas \cite{thm27} and Fermi \cite{frm27} and then revived by Hohenberg, Kohn and Sham \cite{hk64, ks65} and, from a variational point of view by Levy \cite{Lev79} and Lieb \cite{lieb83}. 
The DFT looks at the $N$-point probability distribution of electrons $\rho$ (also known as charge density) as the main variable, replacing the wave function $\psi$ by
$$\rho(x)=\sum_{s_i=\pm1}\int_{\R^{3(N-1)}}|\psi(x,s_1,x_2,s_2,\dots,x_N,s_N)|^2\,dx_2,\dots,dx_N.$$
The usual choice for the definition of the electrons density is $\tilde \rho = N \rho$
	so that the approximations for the kinetic energy and the potential ({\em i.e.} electron-nuclei interaction) terms are respectively given by
\[\begin{split}
&N\, K\, \hbar^2\int\frac{|\nabla\rho|^2}{\rho}\,dx\qquad\hbox{(kinetic)};\\
&N\,\int V(x)\rho\,dx\qquad\hbox{(potential)}.
\end{split}\]
where all the integrals with no domain of integration explicitly defined are intended on $\R^3$. Here we choose to keep $\rho$ as a probability density, and then divide the whole energy by $N$. The constant $K$ depends on the case considered (bosonic or fermionic) see Theorems 1.1 and 1.2 of \cite{lieb83}; in the rest of the paper it will be normalized to $1$.

Concerning the approximation for electron-electron correlation, the considered term can be expressed using the multimarginal mass transport theory (see for instance \cite{bdpgg12,cfk13}) as detailed below.

The Kohn-Sham model is an approximation of the wave functional \eqref{wavef} and aims to determine the probability to find $N$ electrons in a given position once the positions $X_k$ and the charges $Z_k$ of $M$ nuclei are given. The probability $\rho$ is obtained through the minimization of a suitable functional
\be\label{Feps}
F_\eps(\rho)=\eps T(\rho)+bC(\rho)-U(\rho)
\ee
where $\eps$ is a small parameter depending on the Planck constant $\hbar$, $b$ is a given positive constant ($b=1$ in the original Kohn-Sham model, $b=1/N$ with our normalization for $\rho$), and the three terms appearing are defined as follows :
\begin{itemize}
\item The {\it kinetic energy} $T$ is of the form
$$T(\rho)=\int\frac{|\nabla\rho|^2}{4\rho}\,dx=\int|\nabla\sqrt\rho|^2\,dx.$$
\item The {\it correlation term} $C$ is given by means of the multimarginal mass transport functional
\be\label{defC}
C(\rho):=\inf\left\{\int_{\R^{3N}}c(x_1,\dots,x_N)\,dP(x_1,\dots,x_N)\, :\, \forall i=1,\dots,N, \; \pi^\#_iP=\rho\, \right\},
\ee
where $P$ is a probability on $\R^{3N}$, $\pi_i$ is the projection map from $\R^{3N}$ on its $i$-th factor $\R^3$, $\#$ denotes the push forward operator defined for a map $f$ and a measure $\mu$ by
$$f^\#\mu(E)=\mu\big(f^{-1}(E)\big),$$
and $c$ is the Coulomb correlation function
$$c(x_1,\dots,x_N)=\sum_{1\le i<j\le N}\frac{1}{|x_i-x_j|}\;.$$
\item The {\it potential term} $U$ is of the form
$$U(\rho)=\int V(x)\,d\rho(x)$$
being $V(x)$ the Coulomb potential
$$V(x)=\sum_{1\le k\le M}\frac{Z_k}{|x-X_k|}\,.$$
We shall also write $U$ in the form
$$U(\rho)=\sum_{k=1}^M Z_k\,U_{X_k}(\rho)\qquad \text{where}\qquad 
U_{X_k}(\rho):=\int\frac{1}{|x-X_k|}\,d\rho(x)\,.$$
%and where $Z=\sum_{1\le k\le M} Z_k$ is the total charge of the nuclei.
\end{itemize}
%If we set
%$$U_{X_k}(\rho):=\int\frac{1}{|x-X_k|}\,d\rho(x),$$
%we may write
%$$U(\rho)=\sum_{k=1}^M Z_k U_{X_k}(\rho).$$

When the Coulomb correlation term $C$ above is chosen the theory is usually called Strongly Correlated Electrons Density Functional Theory (SCE-DFT); it was started and developed since the late 90's (see for instance \cite{sei99, spl99, sggs07}) and the connection with optimal transport was made in \cite{bdpgg12, cfk13}.

Our goal is to describe the behavior, as $\eps\to0$, of the asymptotically minimizing sequences $\{\rho_\eps\}_\eps$ of the functionals $F_\eps$ in \eqref{Feps}; we will see that, as $\eps\to0$, the minimal values of $F_\eps$ tend to $-\infty$ as $-1/\eps$, so it is convenient to consider the rescaled functionals
\be\label{Geps}
G_\eps(\rho)=\eps F_\eps(\rho)=\eps^2 T(\rho)+\eps bC(\rho)-\eps U(\rho)
\ee
that have the same minimizers as $F_\eps$ (we refer to Definition \ref{defminim} for a precise statement). It turns out that, by a scaling property of the above functionals, minimizing $G_\eps$ is equivalent to minimize the functional $F_1$ with nuclei at positions $\tilde X_i = X_i / \eps$ (see \eqref{Dbond}) : in this respect the asymptotic study of $G_\eps$ corresponds to the study of the so-called {\it dissociating bond problem} (see in particular Remark \ref{friesecke} and \cite{cfm14,mumgg14}).
The goal is then to characterize the $\Gamma$-limit $G$ of $G_\eps$ with respect to the weak* convergence of measures. In this way, the minimizers $\rho_\eps$ of $G_\eps$ (or equivalently of $F_\eps$) will tend to minimizers of $G$ in the weak* convergence of measures. Since we are on the whole space $\R^3$, the weak* convergence in principle does not preserve the total mass, so we could expect that the limits $\rho$ of $\rho_\eps$ are not anymore probabilities and only satisfy the inequality $\int d\rho\le1$.

The analysis of the limit functional $G$ is then very important and the ultimate goal is to characterize $G$ explicitly in terms of the data. This would allow to determine the measures $\rho$ that minimize $G$ and, by consequence, a deep information on the minimizers $\rho_\eps$ of $F_\eps$. For instance, an important issue is to establish if the optimal $\rho$ for $G$ consists of a sum of Dirac masses located at the points $X_k$, that is
$$\rho=\sum_{k=1}^M\alpha_k\delta_{X_k}\qquad\hbox{with $\alpha_k\ge0$ and }\sum_{k=1}^M\alpha_k=1.$$
In this case the minimal value of the functional $G$ would depend on the coefficients $\alpha_k$ and on the charges $Z_k$:
$$G(\rho)=\gamma(Z_1,\dots,Z_M,\alpha_1,\dots,\alpha_M).$$
In principle the function $\gamma$ above could be very involved, mixing the data in a very intricate way; a better situation would occur if
\be\label{better}
\gamma(Z_1,\dots,Z_M,\alpha_1,\dots,\alpha_M)=\sum_{k=1}^M g(Z_k,\alpha_k)
\ee
where the function $g$ can be deduced through the solution of an {\it auxiliary problem} with only one nucleus. This seems to be the case, even if our proofs are complete only in the case $N\le2$ (for any number $M$ of nuclei).

Nevertheless, several points on the asymptotic analysis of the functionals $F_\eps$ can be achieved in full generality. Due to some technical difficulties, mostly related to the analysis of the correlation term $C$, we are able to obtain a complete characterization of the $\Gamma$-limit functional $G$ only when the number $N$ of electrons is at most 2 (for any number $M$ of nuclei). In this case, we use a concept of {\it fractional transport cost}, which allows us to obtain an explicit representation of the $\Gamma$-limit functional $G$. This will imply that the optimal measures $\rho$ for $G$ are actually probabilities, and so the convergence of $\rho_\eps$ to $\rho$ is in the narrow sense. In addition (see Theorem \ref{explform}), formula \eqref{better} is shown to hold, together with an expression of the function $g$ that can be deduced by means of an auxiliary variational problem.

%\blue{Two remarks concerning the literature are due. The $\Gamma$-limit of the first two terms of $F_\eps$ (i.e. $\eps T(\rho)+bC(\rho)$) is $bC(\rho)$ and can be easily deduced from the study of the semi-classical limit of the so-called Hohenberg-Kohn energy or Levy-Lieb functional, which has been considered in several recent papers (see for instance \cite{bdp17, cfk13, cfk17, lew17}). However, since the potential term is not continuous, the limit of the full functional $F_\eps$ cannot be deduced and, in fact, requires a different analysis.}

The characterization of the $\Gamma$-limit of $F_\eps$ as $\eps\to0$ is related to the semi-classical limit of the so-called Hohenberg-Kohn energy or Levy-Lieb functional, which has been considered in several recent papers (see for instance \cite{bdp17, cfk13, cfk17, lew17}). However, since the potential term $U(\rho)$ is not continuous with respect to the weak* convergence, this $\Gamma$-limit cannot be simply deduced from the one of $\eps T(\rho)+bC(\rho)$, which reduces to the lower semicontinuous envelope $b\overline C(\rho)$.

Here is the plan of the paper.
In Sections \ref{sksmodel} and \ref{sresca} we introduce the notations used in the following and show some basic properties of the functionals $F_\eps$ and of its three components. We also introduce the functionals $G_\eps$ together with some of their asymptotic properties. Section \ref{sprel} is devoted to the non-interacting case $b=0$, in which the electron-electron correlation term is not present. This simplifies a lot the analysis and an explicit characterization of the $\Gamma$-limit functional is obtained in this case. In Section \ref{two-interact} we treat the case with the electron-electron interaction term $C(\rho)$ and we provide, in the case $N=2$, a general expression of the $\Gamma$-limit functional $G$. Finally, in the last section of the paper we discuss about the case $N>2$ and some open issues, concluding the paper with some comments about our future work program.

%%%%%%%%%%%%%%%%%%%%%%%%%%%%%%%%%%%%%%%%%%%%%%%%%%
\section{Introductory results on the correlation term $C$}\label{sksmodel}
%%%%%%%%%%%%%%%%%%%%%%%%%%%%%%%%%%%%%%%%%%%%%%%%%%

In the following by $\PP$ we denote the class of Borel probabilities in $\R^3$ and by $\PP^-$ the class of Borel subprobabilities on $\R^3$, that is Borel measures $\mu$ with $\int d\mu\le1$. By $\weak$ and $\rightharpoonup$ we respectively denote the weak* convergence and the narrow convergence on $\PP$ and on $\PP^-$, and by $\delta_X$ the Dirac mass at the point $X$. We also indicate by $\|\rho\|$ the quantity $\int d\rho$.

By $T(\rho)$, $C(\rho)$, $U(\rho)$ we denote the functionals detailed in the Introduction, representing respectively the kinetic energy, the correlation term, and the potential term of a density $\rho \in \PP$. As noted in the introduction, the absolutely minimizing sequences for the functionals $F_\eps$ (or $G_\eps$) may only weak* converge to elements in $\PP^-$, it is then convenient to extend the functionals $T(\rho)$, $C(\rho)$ and $U(\rho)$ to any non-negative bounded measure (in particular on $\PP^-$) by $1$-homogenity.
We shall denote these $1$ homogeneous extensions by $\Tcal(\rho)$, $\Ccal(\rho)$ and $\Ucal(\rho)$ respectively : for $\Tcal$ and $\Ucal$, this extension process obviously leads to the same expression, and these two functionals are lower semicontinuous on $\PP^-$ for the weak* convergence. For the correlation term $\Ccal(\rho)$ this extension
is also obtained through the same expression \eqref{defC} and reads
\[ %\label{Cextend}
\Ccal(\rho):=\inf\left\{\int_{\R^{3N}}c(x_1,\dots,x_N)\,dP(x_1,\dots,x_N)\ :\ \pi^\#_iP=\rho, \ \forall i=1,\dots,N\right\}
\]
where we note that the transport plans $P$ are non-negative Borel measures on $R^{3N}$ with total mass $\Norm{P} = \Norm{\rho}$.

While most of the readers are probably familiar with the first and third terms $T$ and $U$ of $F_\eps$, we believe that the following results are useful to understand the correlation term $C$.

\begin{prop}\label{ntensor}
For every probability $\rho$ on $\R^3$ we have
$$C(\rho)\le\frac{N(N-1)}{2}\int_{(\R^3)^2}\frac{1}{|x-y|}\,d\rho(x)d\rho(y).$$
In particular, $C(\rho)$ is finite for every bounded $\rho$ with compact support.
\end{prop}

\begin{proof}
From the definition of $C(\rho)$, we may take $\rho\otimes\rho\otimes\dots\otimes\rho$ as a particular probability $P$, so that
$$C(\rho)\le\int_{\R^{3N}}c(x_1,\dots,x_N)\,d\rho(x_1)\dots d\rho(x_N)$$
and this last integral reduces to
$$\frac{N(N-1)}{2}\int_{(\R^3)^2}\frac{1}{|x-y|}\,d\rho(x)d\rho(y)$$
by the symmetry of the function $c$.
\end{proof}

\begin{prop}\label{variance}
For every probability $\rho$ on $\R^3$ we have
$$C(\rho)\ge\frac{N(N-1)}{4\sqrt{Var(\rho)}}\;,$$
where $Var(\rho)$ denotes the variance of the probability $\rho$. In particular, when $\rho$ is a Dirac mass, then $Var(\rho)=0$ and we recover that $C(\rho)=+\infty$.
\end{prop}

\begin{proof}
Let $P$ be a probability on $\R^{3N}$ with $\pi^\#_iP=\rho$ for all $i=1,\dots,N$ and let $\gamma=(\pi_i \times \pi_j)^\#P$ be the projection of $P$ on the product $\R^3 \times \R^3$. Then we have
\[1\le\bigg(\int_{(\R^3)^2}\frac{1}{\sqrt{|x_i-x_j|}}\sqrt{|x_i-x_j|}\,d\gamma\bigg)^2\le\int_{(\R^3)^2}\frac{1}{|x_i-x_j|}\,d\gamma\int_{(\R^3)^2}|x_i-x_j|\,d\gamma\;.\]
Without loss of generality we assume $Var(\rho) < +\infty$ (otherwise there is nothing to prove), so that $\rho$ has a finite expectation $\mathbb{E}(\rho)$ and we can write
\[\begin{split}
\int_{(\R^3)^2}|x_i-x_j|\,d\gamma&\le\int_{(\R^3)^2}\big(|x_i-\mathbb{E}(\rho)|+|\mathbb{E}(\rho)-x_j|\big)\,d\gamma\\
&\le2\int|x-\mathbb{E}(\rho)|\,d\rho(x)
%\; \le2\bigg(\int|x-|^2\,d\rho(x)\bigg)^{1/2}
\; \le2\sqrt{Var(\rho)}\;.
\end{split}\]
Thus
$$\int_{(\R^3)^2}\frac{1}{|x_i-x_j|}\,d\gamma\ge\frac{1}{2\sqrt{Var(\rho)}}\;.$$
Summing on all $1\le i<j\le N$ we obtain
$$\int_{\R^{3N}}c(x_1,\dots,x_N)\,dP(x_1,\dots,x_N)\ge\frac{N(N-1)}{4\sqrt{Var(\rho)}}$$
as required.
\end{proof}

\begin{rema} \label{remaC}
By $1$-homogeneity, Propositions \ref{ntensor} and \ref{variance} yield
\[
\forall \rho \in \PP^-, \quad 
\frac{N(N-1)\Norm{\rho}^2}{4\sqrt{Var(\rho)}} \le \Ccal(\rho)\le\frac{N(N-1)}{2 \Norm{\rho}}\int_{(\R^3)^2}\frac{1}{|x-y|}\,d\rho(x)d\rho(y).
\]
In particular $0 < \Ccal(\rho)$ for any non zero smooth $\rho \in \PP^-$ with compact support. 
\end{rema}

It turns out that the function $C(\rho)$ is lower semicontinuous with respect to the tight convergence of probability measures. However this is not sufficient for our purposes since sequences with uniformly bounded energy \eqref{Geps} are not tight in general. Therefore we will deal with the weak* convergence and accordingly it is useful to introduce the following lower semicontinuous extension of $C$ to subprobabilities:
\[%\label{barC}
\forall\rho\in\PP^-, \qquad \Cb(\rho):=\inf\left\{\liminf_{n\to+\infty}C(\rho_n)\ :\ \rho_n\in\PP,\ \rho_n\weak\rho\right\}\,.
\]
A natural guess could be that $\Cb(\rho)$ is equal to $\Ccal(\rho)$; however, this is not the case as the following proposition shows.

\begin{prop}\label{C-not-lsc}
The functional $\Ccal$ is convex on $\PP^-$, and the lower semicontinuous envelope $\Cb$ with respect to
the weak* convergence satisfies
\be \label{CpartN}
\Cb(\rho) \leq \min \left\{\Ccal(\mu)\ :\ 0\le\mu\le\rho,\ \int d\mu \ge \frac{N}{N-1}\left[\int d\rho-\frac{1}{N}\right]\right\}.
\ee
\end{prop}

Note that as a corollary of \eqref{CpartN} it holds 
\[
\int d\rho\le\frac{1}{N}\;\Longrightarrow\;\Cb(\rho)=0\qquad\forall\rho\in\PP^-
\]
which together with Remark \ref{remaC} yields that $\Cb \neq \Ccal$.

%In the proof of Proposition \ref{C-not-lsc} we shall use of the following continuity result from \cite{bcdp17} :
%
%
%\begin{theo}[Theorem 3.11 in \cite{bcdp17}]\label{C-continuous}
%Let $\rho_1,\rho_2 \in \Pcal$ with $\rho_1 - \rho_2 \in L^1$ be such that
%\[
%\exists r >0, \; \forall i \in \{1,2\}, \qquad \sup_{x \in \R^3} \rho_i(B(x,r)) < \frac{1}{N(N-1)^2},
%\]
%then there exists a constant $\kappa$ (depending only on $N$ and $r$) such that
%\[
%|C(\rho_1)-C(\rho_2)| \le \kappa \,\|\rho_1-\rho_2\|_{L^1}.
%\]
%\end{theo}

\begin{proof}%[Proof of Proposition \ref{C-not-lsc}]
The convexity of $\Ccal$ follows from the linearity of the constraint $\pi^\#_iP=\rho$ for all $i$.

For the second statement, we first note that when $\int d\rho \le 1/N$, the minimum in \eqref{CpartN} is clearly $0$ and attained for $\mu=0$. Moreover, since $\Ccal$ is $1$-homogeneous the second constraint in \eqref{CpartN} may be replaced by $\int d\mu = \frac{N}{N-1}\left[\int d\rho-\frac{1}{N}\right]$ whenever $\int d\rho\ge 1/N$. Note that this minimum is then attained since the set of measures
$$\left\{\mu\ :\ 0\le\mu\le\rho, \ \int d\mu = \frac{N}{N-1}\left[\int d\rho-\frac{1}{N}\right] \right\}$$
is tight and $\Ccal$ is lower semicontinuous for the narrow convergence. We also note that \eqref{CpartN} obviously holds when $\rho \in \PP$ since $\mu = \rho$ is then the only admissible choice. 

Now let $\rho\in\PP^-$ with $\int \rho < 1$, we first assume that $\int d\rho\ge1/N$. Let $\mu$ such that $0\le\mu\le\rho$, $\int d\mu = \frac{N}{N-1}\left[\int d\rho-\frac{1}{N}\right]$ and $\mu$ optimal in \eqref{CpartN}. We set $\nu = \rho - \mu$. Let $\xi_1=0$, consider $N-1$ distinct vectors $\xi_2,\ldots,\xi_N\in\R^3\setminus\{0\}$, and set for every $n\in\N$
\[
\rho_{n,k}={\tau_{n,k}}^\#\nu,\qquad k\in\{1,\ldots,N\},
\]
where $\tau_{n,k} := \tau_{(n\,\xi_k)}$, being $\tau_\xi$ the translation by $\xi$ defined on $\R^3$ by $\tau_\xi : x \mapsto x+\xi$. We can now set
\[
\forall n, \qquad \rho_n := \rho+\sum_{k=2}^N\rho_{n,k}=\mu + \sum_{k=1}^N\rho_{n,k}
\]
and we note that $\rho_n$ belongs to $\PP$ for any $n$ and $\rho_n \weak \rho$ as $n \to +\infty$.
We also denote by $P$ an optimal plan for $\Ccal(\mu)$, then we set
\[ 
\forall n, \qquad P_n:=P + \sum_{i=1}^N (\tau_{n,\sigma^i(1)}\times\ldots\times\tau_{n,\sigma^i(N)})^\#\nu\;,
\]
where $\sigma$ is the permutation on $\{1,\ldots,N\}$ such that $\sigma(j)=j+1$ and $\sigma(N)=1$. Then for all $n$ the plan $P_n$ satisfies
\[
\forall j \in \{1,\ldots,N\}, \qquad \pi_j^\# P_n = \mu + \sum_{i=1}^N {\tau_{n,\sigma^i(j)}}^\#\nu = \rho_n
\] 
so $P_n$ is admissible for $C(\rho_n)$. We can now estimate
\begin{align*}
C(\rho_n) \; \le \int_{\R^{3N}} c\,dP_n & = \,
\int_{\R^{3N}} c\,dP + \sum_{i=1}^N \int c(\tau_{n,\sigma^i(1)}(x),\ldots,\tau_{n,\sigma^i(N)}(x)) \nu(dx) \\
& = \int_{\R^{3N}} c\,dP + N\sum_{1\le i<j\leq N} \frac{\Norm{\nu}}{n \norm{\xi_i-\xi_j}}
\end{align*}
from which we conclude that
$$\Cb(\rho)\le\liminf_{n\to+\infty}C(\rho_n) \leq \int_{\R^{3N}} c\,dP = \Ccal(\mu).$$
as required.

%We now turn to the general case. First we consider the case of $\rho \in \PP^- \cap L^1$ with $\frac{1}{N} < \int d\rho < 1$ and having non compact support. Let $\mu$ be such that $0 \leq \mu \leq \rho$ and $\int d\mu > \frac{N}{N-1}\left[\int d\rho-\frac{1}{N}\right]$. Consider an increasing sequence of compact sets $(\O_n)_n$ such
%that $\bigcup_n \O_n = \R^3$, then for $n$ large enough there exist $V_n \subset \O_n$ such that $\int d\mu_{\lfloor V_n} \geq \frac{N}{N-1}\left[\int d\rho_{\lfloor \O_n}-\frac{1}{N}\right]$, and since
%$0 \leq \mu_{\lfloor V_n} \leq \rho_{\lfloor \O_n}$ we obtain from the preceding that
%\begin{equation}\label{rho-restrict}
%\Cb(\rho_{\lfloor \O_n}) \leq \Ccal(\mu_{\lfloor V_n})
%\end{equation}
%for $n$ large enough. Since $\rho \in L^1$ and $\mu_{\lfloor V_n} \leq \mu \leq \rho$ for all $n$, we can apply Theorem \ref{C-continuous} to the family $\left(\frac{\mu_{\lfloor V_n}}{\int d\mu_{\lfloor V_n}}\right)_n$ and obtain a uniform constant $\kappa$ such that
%\[
%\forall n, \qquad \left| \frac{\Ccal(\mu_{\lfloor V_n})}{\int d\mu_{\lfloor V_n}}
%- \frac{\Ccal(\mu)}{\int d\mu} \right|
%\le \kappa \,\left\|\frac{\mu_{\lfloor V_n}}{\int d\mu_{\lfloor V_n}}
%-\frac{\mu}{\int d\mu} \right\|_{L^1}.
%\]
%Since $\int d\mu_{\lfloor V_k} \to \int d\mu$ we conclude that $\Ccal(\mu_{\lfloor V_n}) \to \Ccal(\mu)$ and we may pass to the limit in \eqref{rho-restrict} and use the semicontinuity of $\Cb$ to conclude $\Cb(\rho) \leq \Ccal(\mu)$. Then taking the infimum over $\mu$ yields \eqref{CpartN} in this case.
%
%The case when $\rho \notin L^1$ follows ...

It remains to treat the case $\int d\rho \leq \frac{1}{N}$. It follows from the preceding that $\Cb(\rho) = 0$ for any $\rho$ such that $\int d\rho = \frac{1}{N}$, then by weak* lower semicontinuity of $\Cb$ and by approximation this also holds for any $\rho \in \PP^-$ with $\int d\rho \leq \frac{1}{N}$, which finishes the proof of \eqref{CpartN} in this case.
\end{proof}

The full characterization of $\Cb$ on $\PP^-$ is a quite involved problem, in particular it is an open
problem whether equality holds in \eqref{CpartN} for $N \geq 3$ but holds true for $N=2$.

\begin{prop}\label{prop-Cpart}
In the case $N=2$ it holds
\be\label{Cpart}
\Cb(\rho)=\min \left\{C(\mu)\ :\ 0\le\mu\le\rho,\ \int d\mu \ge 2 \int d\rho-1\right\}\qquad\forall\rho\in\PP^-.
\ee
\end{prop}

%In the above formula, $C$ is defined on $\PP^-$ by \eqref{Cextend}, and since this functional is $1$-homogeneous the constraint $\int d\mu\ge2\int d\rho-1$ may be replaced by $\int d\mu=2\int d\rho-1$ whenever $\int d\rho\ge 1/2$, otherwise the minimum is clearly $0$ as Proposition \ref{C-not-lsc} indicates. Note that this minimum is attained since the set of measures
%$$\left\{\mu\ :\ 0\le\mu\le\rho,\ \int d\mu=2\int d\rho-1\right\}$$
%is tight and $C$ is lower semicontinuous for the narrow convergence.

\begin{proof}
%As in the proof of Proposition \ref{C-not-lsc} above, we assume without loss of generality that $\rho$ has compact support and is such that $\int d\rho\ge1/2$. Let $\mu\in\PP^-$ such that
%$$0\le\mu\le\rho\qquad\hbox{and}\qquad\int d\mu=2\int d\rho-1.$$
%Let then $\nu=\rho-\mu$, take $\xi\in\R^3\setminus\{0\}$ and consider the sequence of probabilities
%$$\rho_n=\mu+\nu+\tau_{n\xi}^\#\nu=\rho+\tau_{n\xi}^\#\nu\,,$$
%then $\rho_n\weak\rho$. Moreover, being $P$ an optimal transport plan for $C(\mu)$ with marginals $\mu$, we can define the transport plan $P_n$ with marginals $\rho_n$ by
%$$P_n=P+\nu\otimes \frac{\tau_{n\xi}^\#\nu}{\|\tau_{n\xi}^\#\nu\|}+\tau_{n\xi}^\#\nu\otimes\frac{\nu}{\|\nu\|},$$
%(where we note that $\|\tau_{n\xi}^\#\nu\|= \|\nu\|$), so that one has
%\[
%C(\rho_n)\le\int_{(\R^3)^2}c\,dP+\frac{2}{\dist\big(\spt(\nu),\spt(\tau_{n\xi}^\#\nu)\big)}
%\]
%for $n$ large enough. We thus obtain
%\[
%\Cb(\rho)\le\liminf_{n\to+\infty}C(\rho_n)\le C(\mu).
%\]
%Taking the infimum among all admissible $\mu$ and using the $1$-homogeneity of $C$ yields
%$$\Cb(\rho)\le\inf\left\{C(\mu)\ :\ 0\le\mu\le\rho,\ \int d\mu\ge2\int d\rho-1\right\}.$$
In order to obtain the reverse inequality of \eqref{CpartN}, consider a sequence $\rho_n$ in $\PP$ weakly* converging to $\rho \in \PP^-$. For each $n$ we denote by $P_n$ an optimal plan for $\rho_n$, then we may assume that $P_n\weak P$ for some non-negative Borel measure $P$ over $(\R^3)^2$, with marginals $\pi_1^\# P = \pi_2^\# P = \mu$ for some $\mu \in \PP^-$. Then one gets
\[\begin{split}
\liminf_{n\to+\infty}C(\rho_n)&=\liminf_{n\to+\infty}\int_{(\R^3)^2} c(x_1,x_2)\,P_n(dx_1,dx_2)\\
&\ge\int_{(\R^3)^2} c(x_1,x_2)\,P(dx_1,dx_2)\ge C(\mu)
\end{split}\]
We first claim that $\mu\le\rho$: indeed, let $\phi,\psi\in \C_c^\infty(\R^3)$ be non-negative and $0\le\psi\le1$, then we have
\[\begin{split}
\langle\rho,\phi\rangle&=\lim_{n\to+\infty}\int_{(\R^3)^2}\phi(x)P_n(dx,dy)\\
&\ge\lim_{n\to+\infty}\int_{(\R^3)^2}\phi(x)\psi(y)P_n(dx,dy)\\
&=\int_{(\R^3)^2}\phi(x)\psi(y)P(dx,dy).
\end{split}\]
Now letting $\psi\nearrow1$ we obtain by the monotone convergence theorem that
\[
\langle\rho,\phi\rangle\ge\int_{(\R^3)^2}\phi(x)\,P(dx,dy)=\langle\mu,\phi\rangle
\]
and, since this is true for any non-negative test function $\phi$, we get the claim.

It remains to prove that $\int d\mu\ge2\int d\rho-1$. For this, consider a ball $B_R$ centered at the origin such that $\rho(\partial B_R)=0$; from what seen above this implies $\mu(\partial B_R)=P(\partial B_R^2)=0$. Then for all $n\in\N$ one has
\[
\rho_n(B_R)=P_n(B_R^2)+P_n(B_R\times B_R^c)\le P_n(B_R^2)+\rho_n(B_R^c)=P_n(B_R^2)+1-\rho_n(B_R)
\]
so that, passing to the limit as $n\to\infty$ gives
$$2\rho(B_R)-1\le P(B_R^2)\le\mu(B_R),$$
and then the claim follows by letting $R$ go to $+\infty$.
\end{proof}

\begin{rema}\label{CbN2}
As a corollary of Propositions \ref{C-not-lsc} and \ref{prop-Cpart}, it appears that when $N=2$ one has $\Cb(\rho) >0$ if and only if $\int d\rho\ge1/2$: indeed, if $\int d\rho\ge1/2$, since the minimum giving $\Cb(\rho)$ is attained for some $\mu$ with $\int d\mu = 2\int d\rho-1 >0$, we obtain $\mu \neq 0$ so that $\Cb(\rho)=C(\mu)>0$.
\end{rema}

%Due to the lack of compactness for the tight convergence, it is not clear whether the functionals $F_\eps$ and $G_\eps$ admit minimizers in $\PP$. As we are interested in the asymptotics as $\eps \to 0$, we will use the following asymptotic notion of minimizer.
%
%\begin{defi}\label{defminim}
%We say that $\{\rho_\eps\}_\eps\in\PP$ is asymptotically minimizing for $G_\eps$ (respectively for $F_\eps$) whenever
%$$\lim_{\eps\to0}\left[G_\eps(\rho_\eps)-\inf_{\PP}G_\eps\right]=0\qquad\left(\hbox{respectively }\lim_{\eps\to0}\frac{F_\eps(\rho_\eps)-\inf_{\PP} F_\eps}{\eps}=0\right).$$
%\end{defi}
%
%Let us emphasize that such a sequence of asymptotic minimizers may converge weakly${}^*$ to an element of $\PP^-$.

%%%%%%%%%%%%%%%%%%%%%%%%%%%%%%%%%%%%%%%%%%%%%%%%%%%%%%%%%%%%%%%%%%%%%%%%%%%%%%%%%%%%%%%
\section{Basic inequalities, properties of the sequence and rescaling}\label{sresca}
%%%%%%%%%%%%%%%%%%%%%%%%%%%%%%%%%%%%%%%%%%%%%%%%%%%%%%%%%%%%%%%%%%%%%%%%%%%%%%%%%%%%%%%
We now get back to the preliminary analysis of $F_\eps$. As proved by Lieb (Theorem 1.1 and Theorem 1.2 of \cite{lieb83}) the correct space where to minimize the functional $F_\eps$ is
$$\HH=\left\{\rho\in L^1(\R^3)\ :\ \rho\ge0,\ \int\rho\,dx =1,\ \sqrt\rho\in H^1(\R^3)\right\}.$$

\begin{prop} For every $\eps >0$ the functional $F_\eps$ is convex.
\end{prop}
\begin{proof}
To prove convexity we look separately at the three terms which compose the functional $F_\eps$. The first term is convex with respect to the pair $(\rho,\nabla \rho)$ thanks to the convexity of the function $(s,v)\mapsto|v|^2/s$ on $\R^+\times\R^d$. We notice that the kinetic energy can be also written in the form
$$T(\rho)=\int|\nabla\sqrt\rho|^2\,dx.$$

%Concerning the transport cost $C(\rho)$, its dual expression gives
%$$C(\rho)=\sup\left\{N\int u(x)\,d\rho(x)\ :\ \sum_{i=1}^N u(x_i)\le\sum_{1\le i<j\le N}\frac{1}{|x_i-x_j|}\right\},$$
%which shows that $C(\rho)$ is the supremum of a family of linear functionals., and hence is convex.
The transport cost $C(\rho)$ defined in \eqref{defC} is a linear problem with respect to the probability $P$ on $\R^{3N}$, with linear constraints $\pi_i^\# P = \rho$, so that it is convex. The term $U(\rho)$ is linear with respect to $\rho$.
\end{proof}

The different behavior of the three terms with respect to the action of homotheties on the probability measures has a relevance in the study of the $\Gamma$-limit.

\begin{defi}%\label{rhos}
For every probability measure $\rho \in \PP$ and every $s>0$ we set
\[
\rho^s = h_s^\# \rho
\]
where $h_s:x\mapsto x/s$.
\end{defi}
 
\begin{rema}
Note in particular that if $\rho$ has a compact support then the support of $\rho^ s$ is $h_s\big(supp(\rho)\big)$. We shall use in the following sections that $\rho^s\rightharpoonup\delta_0$ as $s\to+\infty$, which straightly follows from the definition. Also notice that the map $\rho\mapsto\rho^s$ has for inverse $\rho\mapsto\rho^{1/s}$. Finally, since the elements $\rho\in\PP$ for which the functionals $G_\eps$ are defined are in $L^1$, it is interesting to note that in this case the probability $\rho^s$ is also in $L^1$ with density
$$\rho^s(x)=s^3\rho(sx)\,.$$
\end{rema}

%Clearly $\rho^s$ and $\rho^s_{X_i}$ are still probabilities and the maps $\rho\mapsto\rho^s$ and $\rho\mapsto\rho^s_{X_i}$ are invertible.
When $M=1$ and $X_1=0$, a simple calculation shows that the three terms in $F_\eps$ scale as follows:
$$T(\rho^s)=s^2T(\rho),\qquad C(\rho^s)=sC(\rho),\qquad U(\rho^s)=U_0(\rho^s)=sU_0(\rho).$$
Therefore, minimizing the functional \eqref{Feps} with respect to $\rho^s$ leads to minimize the quantity
$$\eps s^2T(\rho)+s\big(bC(\rho)-U(\rho)\big).$$
A first minimization with respect to the variable $s\ge0$ reduces the problem $\displaystyle \inf_\rho F_\eps$ to minimize with respect to $\rho$ the ratio
\[
-\frac{1}{4\eps}\frac{\big(U(\rho)-bC(\rho)\big)_+^2}{T(\rho)}
\]
where $(\cdot)_+$ denotes the positive part function. Then
\be\label{ratio}
\inf_{\rho \in \Pcal} F_\eps(\rho)=-\frac{K}{\eps}\;,
\ee
where 
$$K=\sup_{\rho \in \Pcal} \frac{\big(U(\rho)-bC(\rho)\big)_+^2}{4T(\rho)}\;.$$
This last supremum is finite as a consequence of the inequality $\big(U-bC\big)_+^2\le U^2$ and of Lemma \ref{l1} below applied to $u^2=\rho$ for any probability $\rho$ in the domain of $F_\eps$.

When $M>1$ the potential term $U$ does not have anymore the scaling property above so we use the following estimate (setting $Z=\sum_{1\le k\le M}Z_k$):
$$F_\eps(\rho)=\sum_{1\leq k\leq M}\frac{Z_k}{Z}\left[\eps T(\rho)+bC(\rho)-\int\frac{Z\rho(x)}{|x-X_k|}dx\right],$$
so that
\begin{align*}
\sum_{1\le k\le M}\frac{Z_k}{Z}\inf_\rho\left[\eps T(\rho)+bC(\rho)-\int\frac{Z\rho(x)}{|x-X_k|}dx\right] & \le \inf_\rho F_\eps(\rho)\\
&\le\inf_\rho\left[\eps T(\rho)+bC(\rho)-\int\frac{Z_1\rho(x)}{|x-X_1|}dx\right]\;.
\end{align*}
It follows that, for suitable positive constants $K_1$ and $K_2$, we have
$$-\frac{K_1}{\eps}\le\inf_\rho F_\eps(\rho)\le-\frac{K_2}{\eps}\;.$$
In conclusion the minimal values of $F_\eps$ tend to $-\infty$ with order $1/\eps$. This justifies the introduction of the rescaled functionals
$$G_\eps(\rho)=\eps F_\eps(\rho).$$
Due to the lack of compactness for the tight convergence, it is not clear whether the functionals $F_\eps$ and $G_\eps$ admit minimizers in $\PP$. As we are interested in the asymptotics as $\eps\to0$, we will use the following asymptotic notion of minimizer.

\begin{defi}\label{defminim}
We say that $\{\rho_\eps\}_\eps\in\PP$ is asymptotically minimizing for $G_\eps$ (respectively for $F_\eps$) whenever
$$\lim_{\eps\to0}\left[G_\eps(\rho_\eps)-\inf_{\PP}G_\eps\right]=0\qquad\left(\hbox{respectively }\lim_{\eps\to0}\frac{F_\eps(\rho_\eps)-\inf_{\PP} F_\eps}{\eps}=0\right).$$
\end{defi}

Let us emphasize that such a sequence of asymptotic minimizers converge (up to subsequences) weakly* to an element of $\PP^-$.
%In this way the minimizing sequences (in the sense of Definition \ref{defminim}) do not change and 
We may now apply to the family $G_\eps$ the $\Gamma$-convergence theory in order to identify the 
$\Gamma$-limit functional $G$ and so, as a consequence, the behaviour of the asymptotically minimizing sequences, that will converge to minimizers of $G$.

The functional $G_\eps$ has also a physical interest by itself and deserves to be written explicitly
$$G_\eps(\rho)=\eps^2 T(\rho)+\eps bC(\rho)-\eps U(\rho).$$
The same homogeneities of the terms in the functional $F_\eps$ allow us to rewrite
\be\label{Dbond}
G_\eps(\rho)=T(\rho^\eps)+ bC(\rho^\eps)-\sum_{1\leq k\leq M}\int\frac{Z_k \rho^\eps (x)}{|x-X_k/\eps|}\,dx.
\ee
Then letting $\eps$ go to $0$ is equivalent to let the distance between the nuclei go to $+\infty$. For this reason, when considering a molecule, the limit as $\eps\to0$ of $G_\eps$ models the dissociation of chemical bonds between the atoms composing the molecule.

In the following we denote by $G^+$ and by $G^-$ respectively the $\Gamma$-limsup and the $\Gamma$-liminf of the family $G_\eps$. Since the space $\PP^-$ endowed with the weak* convergence is metrizable and compact, by the general theory of the $\Gamma$-convergence (see for instance \cite{dalm}) we have that a subsequence of $G_\eps$ (that we still continue to denote by $G_\eps$) $\Gamma$-converges to some functional $G$. If we are able to fully characterize this limit functional $G$ independently of the subsequence, we obtain that the full family $G_\eps$ is $\Gamma$-convergent to $G$. Therefore, in the following we may assume that $G_\eps$ $\Gamma$-converges to some functional $G$ and we concentrate our efforts in obtaining a characterization of $G$ in terms of the data only.

Since in general weak* limits of sequences of probabilities only belong to $\PP^-$, we consider $G^-$, $G^+$, and $G$ as defined on $\PP^-$. As a basic consequence of $\Gamma$-convergence theory (see \cite{att84}) we have the following result.

\begin{prop} We have $G^-\le G^+$; moreover the functionals $G^-$ and $G^+$ are both weakly* lower-semicontinuous, and $G^+$ is convex.
\end{prop}

If we forget about the electron-electron interaction, i.e. the optimal transport term $C$ in $G_\eps$, we obtain an estimate from below and at the same time an easier problem to work with. We will refer to this as the {\it``non-interacting case''} and the corresponding functionals will be denoted as
$$G_\eps^0=\eps^2 T(\rho)-\eps U(\rho).$$

We characterize first a wide space on which $G=G^+=G^-=0$.

\begin{lemm}\label{l1}
There exists a constant $\kappa$ such that for every domain $\O$ (bounded or not) we have
$$\bigg[\int_\O\frac{u^2}{|x|}\,dx\bigg]^2\le \kappa \int_\O u^2\,dx\int|\nabla u|^2\,dx\qquad\forall u\in H^1(\R^3).$$
\end{lemm}

\begin{proof}
By using the embedding of $H^1(\R^3)$ into $L^6(\R^3)$ we have for every $\delta>0$
\[\begin{split}
\int_\O\frac{u^2}{|x|}\,dx&\le\int_{\O\setminus B_\delta}\frac{u^2}{\delta}\,dx+\int_{B_\delta}\frac{u^2}{|x|}\,dx\\
&\le\frac{1}{\delta}\int_\O u^2\,dx+\|u\|^2_{L^6(\R^3)}\big\|1/|x|\big\|_{L^{3/2}(B_\delta)}\\
&\le\frac{1}{\delta}\int_\O u^2\,dx+\kappa\, \delta\int|\nabla u|^2\,dx.
\end{split}\]
Optimizing with respect to $\delta$ gives
$$\int_\O\frac{u^2}{|x|}\,dx\le2\bigg[\kappa \int_\O u^2\,dx\int|\nabla u|^2\,dx\bigg]^{1/2}$$
as required.
\end{proof}

\begin{prop}\label{l2}
For every probability $\rho$ we have
\be\label{ineq2}
-\frac{\kappa\,M}{4}\sum_{1\le k\le M}Z^2_k\rho(\{X_k\})\le G^-(\rho)\le G^+(\rho)\le0,
\ee
where $\kappa$ is given in Lemma \ref{l1}. In particular $G^- (\rho)=G^+(\rho)=0$ for every probability $\rho$ 
that does not charge any of the points $X_k$.
\end{prop}

\begin{proof}
By Proposition \ref{ntensor} we have for every smooth $\rho$ with compact support
$$G_\eps(\rho)=\eps^2T(\rho)+\eps bC(\rho)-\eps U(\rho)\le K\eps$$
for a suitable constant $K$ depending on $\rho$. Therefore
$$G^+(\rho)\le\lim_{\eps\to0}G_\eps(\rho)\le0,$$
and the last inequality in \eqref{ineq2} follows by approximation and by the lower semicontinuity of $G^+$.

Let now $\rho_\eps$ be a generic sequence weakly* converging to $\rho$; since the transport cost $C(\rho_\eps)$ is nonnegative we have, setting $u_\eps^2=\rho_\eps$,
$$G_\eps(\rho_\eps)\ge\eps^2\int|\nabla u_\eps|^2\,dx-\eps\sum_{1\le k\le M}Z_k\left[\int_{B_\delta(X_K)}\frac{u_\eps^2}{|x-X_k|}\,dx+\int_{(B_\delta(X_k))^c}
\frac{u_\eps^2}{|x-X_k|}\,dx\right].$$
By using Lemma \ref{l1} and the fact that $|x-X_k|\ge\delta$ on $B_\delta(X_k)^c$ we obtain
$$G_\eps(\rho_\eps)\ge\eps^2\int|\nabla u_\eps|^2\,dx-\eps\sum_{1\le k\le M}Z_k\left[\kappa\int_{B_\delta(X_K)}u_\eps^2\,dx\int|\nabla u_\eps|^2\,dx\right]^{1/2}-\eps\frac{Z}{\delta}\;.$$
Since $\eps^2 A-\eps B\ge-B^2/(4A)$, the sum of the first two terms in the last line gives
$$G_\eps(\rho_\eps)\ge-\frac{\kappa}{4}\left(\sum_{1\le k\le M}Z_k\bigg[\int_{B_\delta(X_K)}u_\eps^2\,dx\bigg]^{1/2}\right)^2-\eps\frac{Z}{\delta}\;.$$
As $\eps\to0$ we obtain, for every $\delta>0$,
$$G^-(\rho)\ge-\frac{\kappa}{4}\left(\sum_{1\le k\le M}Z_k\Big[\rho(\overline{B_\delta(X_k)})\Big]^{1/2}\right)^2,$$
and finally, as $\delta\to0$,
$$G^-(\rho)\ge-\frac{\kappa}{4}\left(\sum_{1\le k\le M}Z_k\Big[\rho(\{X_k\})\Big]^{1/2}\right)^2\ge-\frac{\kappa\,M}{4}\sum_{1\le k\le M}Z^2_k\rho(\{X_k\})$$
which concludes the proof.
\end{proof}

%%%%%%%%%%%%%%%%%%%%%%%%%%%%%%%%%%%%%%%%%%%%%%%%%%
\section{The non-interacting case}\label{sprel}

In this case, thanks to the absence of the transport term $C$, we are able to identify 
the limit functional $G$ in a complete way. In order to stress the fact that $b=0$ we denote the sequence by $G_\eps^0$ and the limit by $G^0$.

%%%%%%%%%%%%%%%%%%%%
\subsection{The hydrogen atom}%\label{sshydro}
The simplest case is $N=M=1$; in other words we have a single nucleus with charge $Z$ located at a point $X_1$ (that without loss of generality we can take the origin) and a single electron.
In this is situation the non interacting case maintains a physical meaning. The problem \eqref{ratio} then reduces to
$$\inf_{\rho\in\PP}F^0_\eps(\rho)
=\inf_{\rho\in\PP}\big\{\eps T(\rho)-U(\rho)\big\}
=\inf_{\rho\in\PP}\left\{-\frac{1}{4\eps}\frac{\big(U(\rho)\big)^2}{T(\rho)}\right\}
=\frac{Z^2}{\eps}\inf_{\rho\in\PP}\left\{-\frac{\big(U_0(\rho)\big)^2}{4\,T(\rho)}\right\}\,.$$
%the minimization, on the class of probabilities $\rho$, of the functional
%$$-\frac{1}{4\eps}\frac{\big(U(\rho)\big)^2}{T(\rho)}\;.$$
The value of the problem on the right hand side is the first eigenvalue (negative) of the operator
$$-\Delta-\frac{1}{|x|}\;.$$
This is known to be equal to $-1/4$ with eigenfunctions proportional to
$$\rhob=\frac{1}{32\pi} e^{-|x|/2}$$
(see for instance Example 11.10 in \cite{ll97}). Summarizing, in the case $N=M=1$ the minimizer $\rhob_\eps$ of the functional $F^0_\eps$ (or equivalently of the rescaled functional $G^0_\eps$) is equal to
$$\rhob_\eps(x) = \rhob^{1/\eps}(x) = \frac{1}{4\pi(2\eps)^3}e^{-|x|/(2\eps)}\;,$$
and tends, as $\eps\to0$, to the measure $\rhob(x)= \delta_0$ with the minimal values $G_\eps(\rhob_\eps)\to-Z^2/4$. In fact, by applying Theorem \ref{liebG} to this particular case, we infer that the $\Gamma$-limit functional $G^0$ is indentified on $\PP^-$ as
$$G^0(\rho)=-\frac{Z^2}{4}\rho(\{0\}).$$
From the previous discussion we deduce
\be\label{usualtrick}
-\frac{Z^2}{4} \;=\; \inf_{\rho\in\PP}\big\{ T(\rho)- Z\, U_0(\rho)\big\} \;=\; \inf_\rho\left\{\int\frac{|\nabla\rho|^2}{4\rho}\,dx-\int\frac{Z\rho}{|x|}\,dx\right\}.
%=\min_\rho-\frac{1}{4T(\rho)}\left(\int\frac{Z_1\rho}{|x|}\,dx\right)^2.
\ee

%%%%%%%%%%%%%%%%%%%%
\subsection{The general case $N, M\ge 1$}%\label{sbzero}

We start by a localization lemma.

\begin{lemm}[Localization]\label{localize}Let $\rho \in\PP$, and let $\delta>0$. Let $\theta_\delta$ be a smooth cut-off function such that $ \frac{|\nabla \theta_\delta|^2}{\theta}$ is continuous and such that
\be\label{thetad}
\begin{cases}
\theta_\delta=1&\hbox{on } \displaystyle A_\delta:=\bigcup_{i=1}^M B(X_i,\delta),\\
\theta_\delta=0&\hbox{outside } \displaystyle \bigcup_{i=1}^M B(X_i,2\delta),\\
0\le\theta_\delta\le1&\hbox{on }\R^3,\\
\ds \norm{\frac{|\nabla \theta_\delta|^2}{\theta_\delta}-2 \Delta \theta_\delta} \leq K. &
\end{cases}\ee
Let $\nu$ be smooth and compactly supported function such that:
\be\label{nueps}\begin{cases}
\eps \Tcal(\nu) \leq K,\\
\theta_\delta\rho+\nu\in\PP,\\
\dist\big(\spt\nu,\{X_1,\dots,X_M\big)\ge3\delta.
\end{cases}\ee
Then, setting $Z=\sum_{1\le k\le M}Z_k$, we have
$$G^0_\eps(\theta_\delta\rho+\nu)\le G^0_\eps(\rho)+\eps\left( (1+\eps)K+\frac{Z}{\delta} \right).$$
\end{lemm}

\begin{proof} We compute
$$\frac{|\nabla(\theta_\delta\rho)|^2}{\theta_\delta\rho}=\frac{|\theta_\delta\nabla\rho+\rho\nabla\theta_\delta|^2}{\theta_\delta\rho}=\theta_\delta\frac{|\nabla\rho|^2}{\rho}+\rho\frac{|\nabla\theta_\delta|^2}{\theta_\delta}+2\nabla \rho \cdot \nabla \theta_\delta.$$
Concerning the second and third terms we remark that
$$\norm{\int\rho\frac{|\nabla\theta_\delta|^2}{\theta_\delta}+2\nabla\rho\cdot\nabla\theta_\delta\,dx}
=\norm{\int \left(\frac{|\nabla \theta_\delta|^2}{\theta_\delta}-2 \Delta \theta_\delta \right) \rho\, dx} \leq K. $$
Summarizing, we have obtained
$$\int\frac{|\nabla(\theta_\delta\rho)|^2}{4\,\theta_\delta\rho}\,dx\le T(\rho)+K.$$
Since the support of $\nu$ is away from that of $\theta_\delta \rho$ we have that
$$\eps^2 T (\theta_\delta\rho +\nu)=\eps^2\Tcal(\theta_\delta\rho)+ \eps^2 \Tcal(\nu) \leq \eps^2 T(\rho)+ (\eps^2+\eps) K. $$
Similarly,
$$\int V(x)\theta_\delta\rho\,dx\ge\int V(x)\rho\,dx-\int_{A_\delta^c}V(x)\rho\,dx\ge\int V(x)\rho\,dx-\frac{Z}{\delta}\; , $$
and 
$$\int V(x)\nu\,dx\ge 0.$$
Summing up the last inequalities give the desired estimate. 
\end{proof}

\begin{exam}\label{exthetanu}
We give for completeness an example of functions $\theta_\delta$ and $\nu$ satisfying the assumptions of Lemma \ref{localize} above.
We define the real functions 
$$f(t)= 
\begin{cases}
e^{-1/t} & \mbox{if} \ t>0,\\
0 & \mbox{otherwise.}
\end{cases}
\qquad\text{and}\qquad g(t):= \frac{f(t)}{f(t)+f(1-t)}\,.
$$
When $\delta\ll\|X_i-X_j\|$ for $i\ne j$ one may consider
$$\theta_\delta (x)=\sum_{i=1}^M g\left(2-\left(\frac{\|x-X_i\|}{\delta}\right)^2\right).$$
Concerning $\nu$ it is enough to consider any positive function $h$ with smooth $\sqrt{h}$ supported away from the points $X_i$ and set $\nu = \beta h$ with $\beta$ small enough so that $\theta_\delta\rho+\nu\in\Pcal$: in that case the constant $K$ in \eqref{nueps} does not depend on $\eps$.
\end{exam}
	
\begin{lemm}\label{G0ofdelta} We have
$$G^{0+}(\delta_{X_k})\leq -\frac{Z_k^2}{4}.$$
\end{lemm}

\begin{proof}
It is enough to prove the inequality for $X_1$ and $Z_1$ and without loss of generality we may assume $X_1=0$.
%First we remark that, as proved above, for all $\eps>0$
%\begin{equation}\label{usualtrick}
%-\frac{Z_1^2}{4}=\min_\rho\left\{\eps^2\int\frac{|\nabla\rho|^2}{4\rho}\,dx-\eps\int\frac{Z_1\rho}{|x|}\,dx\right\}=
%\min_\rho-\frac{1}{4T(\rho)}\left(\int\frac{Z_1\rho}{|x|}\,dx\right)^2.
%\end{equation}
%To prove the inequality $\le$ 
Consider a generic probability $\eta$ and define
$$\rho_\eps(x)=\eta^{1/\eps}(x)=\frac{1}{\eps^3}\eta(x/\eps).$$
We have that $\rho_\eps$ weakly* converges to $\delta_0$, so that
$$G^{0+}(\delta_0)\le\limsup_{\eps\to0}G_\eps^0(\rho_\eps)=T(\eta)-\int \frac{Z_1\eta}{|x|}\,dx.$$
%Optimizing with respect to $\beta$ gives
%$$G^{0+}(\delta_0)\le-\frac{1}{4T(\eta)}\left(\int\frac{Z_1\eta}{|x|}\,dx\right)^2$$
%and finally, 
Taking the infimum with respect to $\eta$ and using \eqref{usualtrick} gives what required.
\end{proof}

\begin{theo}\label{liebG}
The limit functional $G^0$ exists and is given by the formula
\be\label{eg0}
G^0(\rho)=-\frac14\sum_{1\le k\le M}Z_k^2\rho(\{X_k\}).
\ee
\end{theo}

\begin{proof} We prove that 
\be\label{eg1}
-\frac14\sum_{1\le k\le M}Z_k^2\rho(\{X_k\}) \leq G^{0-}(\rho)\leq G^{0+}(\rho)\leq -\frac14\sum_{1\le k\le M}Z_k^2\rho(\{X_k\}).
\ee
The last inequality in \eqref{eg1} follows by the convexity of the functional $G^{0+}$ and Lemma \ref{G0ofdelta}. Indeed, we can write for every probability $\rho$
$$\rho=\sum_{1\le k\le M}\alpha_k\delta_{X_k}+\alpha_0\frac{\rhop}{\alpha_0}$$
where
$$\alpha_k=\rho(\{X_k\}),\qquad\rhop=\rho\lfloor\big\{\R^3\setminus\cup_{1\le k\le M}\{X_k\}\big\},\qquad\alpha_0=\rhop(\R^3).$$
Since $\alpha_0+\sum_{1\le k\le M}\alpha_k=1$ the convexity of $G^+$ gives
$$G^{0+}(\rho)\le\sum_{1\le k\le M}\alpha_k G^{0+}(\delta_{X_k})+\alpha_0 G^{0+}(\rhop/\alpha_0).$$
By Lemma \ref{G0ofdelta} we have
$$G^{0+}(\delta_{X_k})\leq-\frac14 Z_k^2$$
and, since $G^{0+}(\rho)=0$ whenever $\rho$ does not charge any of the points $X_k$ (see Proposition \ref{l2}), we have $G^{0+}(\rhop/\alpha_0)=0$ so that the desired inequality follows.

In order to prove the first inequality in \eqref{eg1} we have to show that for every $\rho_\eps$ weakly* converging to $\rho$ we have
$$-\frac14\sum_{1\le k\le M}Z_k^2\rho(\{X_k\})\le\liminf_{\eps\to0}G^0_\eps(\rho_\eps).$$
We apply Lemma \ref{localize} to $\rho_\eps$ with $\delta$ small enough and fixed, so that we can replace $\rho_\eps$ by 
$$\theta_\delta \rho_\eps + \nu_\eps = \sum_{k=1}^M \theta_\delta \rho_\eps \lfloor B(X_k, 2 \delta) + \nu_\eps = \sum_{k=1}^M \rho_\eps^k + \nu_\eps$$
where we defined $\rho_\eps^k:=\theta_\delta \rho_\eps \lfloor B(X_k, 2 \delta)$ and $\nu_\eps$ is chosen as in Example \ref{exthetanu} with a fixed function $h$ so that the constant $K$ does not depend on $\eps$. We then have
$$ \liminf_{\eps \to 0} G^0_\eps (\rho_\eps) \geq \liminf_{\eps \to 0} \left(G^0_\eps (\sum_{k=1}^M \rho_\eps ^k + \nu_\eps) - \eps K_\delta \right)=\liminf_{\eps \to 0} \,G^0_\eps (\sum_{k=1}^M \rho_\eps ^k + \nu_\eps),$$
where we can take $K_\delta = 2K+\frac{Z}{\delta}$ for $\eps\le1$ and 
$$ G^0_\eps (\sum_{k=1}^M \rho_\eps ^k + \nu_\eps)= \sum_{k=1}^M \left(\eps^2 \Tcal (\rho_\eps^k) -\eps Z_k U_{X_k} (\rho_\eps^k)\right) + \eps^2 \Tcal (\nu_\eps) - \eps U(\nu_\eps) .$$
Concerning the first $M$ terms we infer from \eqref{usualtrick} that
$$\eps^2 \Tcal (\rho_\eps^k) -\eps Z_k U_{X_k} (\rho_\eps^k)=\left(\eps^2 T \left(\frac{\rho_\eps^k}{\int d \rho^k_\eps}\right) -\eps Z_k U_{X_k} \left(\frac{\rho_\eps^k}{\int d \rho^k_\eps}\right)\right)\int d \rho_\eps^k \geq - \frac{Z_k^2}{4} \int d \rho_\eps^k.$$
As $\eps\to0$ we have
$$\lim_{\eps\to0}\int d\rho_\eps^k=\int_{B(X_k,2\delta)}\theta_\delta\rho\le\rho\big(B(X_k,2\delta)\big).$$ 
The last term $\eps^2 \Tcal (\nu_\eps) - \eps U(\nu_\eps)$ vanishes as $\eps\to0$. Summing up we obtained
$$\liminf_{\eps\to0}G^0_\eps\Big(\sum_{k=1}^M \rho_\eps^k+\nu_\eps\Big)\ge-\sum_{k=1}^M\frac{Z_k^2}{4}\rho\big(B(X_k,2\delta)\big).$$
Letting now $\delta\to0$ gives the desired inequality.
%$$-\sum_{1\le k\le M}\rho(\{X_k\})\frac{Z_k^2}{4}\le\liminf_{\eps\to0}G_\eps(\rho_\eps).$$
\end{proof}

\begin{rema}
Since the correlation term $C$ is non-negative, we obtain from the preceding that
\[
\forall \rho \in \Pcal, \qquad G^-(\rho) \geq G^{0-}(\rho) =-\frac{1}{4}\sum_{1\le k\le M}Z_k^2\rho(\{X_k\}).
\]
which is a more precise lower estimate of $G^-$ than that obtained in Proposition \ref{l2}.
\end{rema}
	
\begin{rema}\label{G0min}
The explicit form of $G^0$ given in \eqref{eg0} allows to directly deduce that
$$\argmin G^0=\left\{\rho=\sum_{i\in I_{max}}\alpha_i\delta_{X_i}\ :\ 0\le\alpha_i\le1,\ \sum\alpha_i=1\right\},$$
where $I_{max}$ denotes the set of indices $i$ such that $Z_i=\max\{Z_1,\dots,Z_M\}$. In other words, the optimal configurations for the limit functional $G^0$ concentrate on the points $X_i$ having the highest nuclei charges $Z_i$. We shall see in Remark \ref{Gmin} below that the situation is somewhat more intricate in the interacting case $b>0$.
\end{rema}

%%%%%%%%%%%%%%%%%%%%%%%%%%%%%%%%%%%%%%%%%%%%%%%%%%%%%%%%%%%%
\section{The interacting case for $N=2$}\label{two-interact}
%%%%%%%%%%%%%%%%%%%%%%%%%%%%%%%%%%%%%%%%%%%%%%%%%%%%%%%%%%%%

The complete characterization of the functional $G$ defined as the $\Gamma$-limit of the functionals $G_\eps$ as $\eps\to0$, in the general case $M\ge1$ and $N\ge1$ seems a very difficult issue, mainly due to the general form of the localization Lemma \ref{localiz} below, that is at the moment unavailable. We then limit ourselves in this section to consider the case $N=2$ that we can handle completely.

For every $\rho\in\PP^-$ we denote by $\rho^\#$ the atomic measure
$$\rho^\#=\sum_{i=1}^M \rho(\{X_i\})\delta_{X_i}$$
where $X_i$ are the (fixed) positions of the nuclei ($i=1,\dots,M$). We also denote by $\rhop$ the measure
$$\rhop=\rho-\rho^\#$$
that does not charge any of the points $X_i$. In other words, $\rho^\#$ and $\rhop$ are respectively the restrictions of $\rho$ to the sets $\{X_1,\dots,X_M\}$ and to its complement.

The following Lemma extends the localization argument of Lemma \ref{localize} in presence of the correlation term $C$.

\begin{lemm}\label{localiz}(Localization)
Let $N=2$, $\rho \in\PP$, and $\delta>0$, and let 
$\theta_\delta$ and $\nu$ be as in Lemma \ref{localize}. Assume further that $\nu=\nu ^1+ \nu ^2$ with $\|\nu ^1\|=\| \nu ^2\|$ and $\spt(\nu^1)\cap\spt( \nu^2)= \emptyset$.
Then we have
$$G_\eps(\theta_\delta\rho+\nu)\le G_\eps(\rho)+\eps\left( (1+\eps)K+\frac{Z+2b}{\delta} +\frac{b}{\dist(\spt(\nu^1),\spt(\nu^2))}\right).$$
\end{lemm}

\begin{proof}
By Lemma \ref{localize} we only need the following estimate of the transport term $C$: 
$$C(\theta_\delta\rho+\nu)\le C(\rho)+\frac{2}{\delta}+ \frac{1}{\dist(\spt(\nu^1),\spt( \nu^2))}.$$
It is convenient to introduce the set $R=\R^3\setminus A_\delta$, being $A_\delta$ defined in \eqref{thetad}).

We denote by $P$ an optimal transport plan for $\rho$, which is also symmetric with respect to a permutation of the variables and we define a new transport plan $\tilde P$ as below.
\[\begin{split}
&\tilde P ^1=\tilde P_{|A_\delta\times A_\delta}=\min\{\theta_\delta(x),\theta_\delta(y)\}P_{|A_\delta\times A_\delta}\;,\\
&\tilde P^2=\tilde P_{|A_\delta\times R}=(\theta_\delta\rho-\pi^\#_1\tilde P^1)\otimes
\frac{\nu}{\|\nu\|}\;,\\
&\tilde P^3=\tilde P_{|R\times A_\delta}=\frac{\nu}{\|\nu\|}\otimes(\theta_\delta\rho_\eps-\pi^\#_2\tilde P^1)\;,\\
&\tilde P^4=
2\,\left(1- \frac{\|\theta_\delta \rho_\eps - \pi_2^\#\tilde P^1\|}{\|\nu\|}\right)
\left( \nu^1 \otimes \nu^2 + \nu^2 \otimes \nu^1 \right)
%\tilde P_{|R\times R}=\frac{1}{\| \nu-\pi^\#_1\tilde P^3\|}(\nu-\pi^\#_1\tilde P^3)\otimes(\nu -\pi^\#_2\tilde P^2)\;.
\end{split}\]
The fact that $\tilde P^2 \geq 0$ (and similarly that $\tilde P^3 \geq 0$) follows from 
$\tilde P^1\le\theta_\delta (x) P_{|A_\delta \times A_\delta}$. To check that $\tilde P^4 \geq 0$ it is enough to compute
\begin{align*}
\|\nu\|- \|\theta_\delta \rho - \pi_2^\#\tilde P^1\|&= 1-\|\theta_\delta \rho\|-(\|\theta_\delta\rho\|-\|\pi_2^\#\tilde P^1 \|)\\
&=1+\|\pi_2^\#\tilde P^1 \| - \|\theta_\delta \rho\|-\|\theta_\delta \rho\|\\
&=\int\left[ 1+ \min \{ \theta_\delta (x), \theta_\delta(y)\} - \theta_ \delta (x)- \theta_\delta(y)\right] d P(x,y) \geq 0
\end{align*}
where the last inequality follows since the integrand is non-negative. Also note that
$$ \nu-\pi^\#_1\tilde P^3 =\left(1- \frac{\|\theta_\delta \rho_\eps - \pi_2^\#\tilde P^1\|}{\|\nu\|}\right) \nu $$
so that $\pi_1^\# \tilde P_4 = \nu-\pi^\#_1\tilde P^3$.
%
%check that the first factor is non-negative and this follows from
%$$ \nu-\pi^\#_1\tilde P^3 =\left(1- \frac{\|\theta_\delta \rho_\eps - \pi_2^\#\tilde P^1\|}{\|\nu\|}\right) \nu $$
%which is a non-negative measure since
%\begin{align*}
%\|\nu\|- \|\theta_\delta \rho - \pi_2^\#\tilde P^1\|&= 1-\|\theta_\delta \rho\|-(\|\theta_\delta \rho\| - \| \pi_2^\#\tilde P^1 \|)\\
%&=1 + \| \pi_2^\#\tilde P^1 \| - \|\theta_\delta \rho\|-\|\theta_\delta \rho\| \\
%&= \int\left( 1+ \min \{ \theta_\delta (x), \theta_\delta(y)\} - \theta_ \delta (x)- \theta_\delta(y)\right) d P(x,y) \geq 0,
%\end{align*}
%and the last inequality follows since the integrand is non-negative.

In order to show that the two marginals of $\tilde P$ coincide with $\theta_\delta\rho+\nu$, since $\tilde P$ is symmetric, it is enough to check the first marginal. Since $R$ is the complement of $A_\delta$ we compute the restriction of the marginal to these two sets. Then
\[\begin{split}
&\mbox{on $A_\delta$ we have}\quad\pi^\#_1\tilde P=\pi^\#_1\tilde P^1+\pi^\#_1\tilde P^2
=\pi^\#_1\tilde P^1+\theta_\delta\rho-\pi^\#_1\tilde P^1
=\theta_\delta\rho\;,\\
&\mbox{on $R$ we have}\quad\pi^\#_1\tilde P=\pi^\#_1\tilde P^3+\pi^\#_1\tilde P^4
=\pi^\#_1\tilde P^3+\nu-\pi^\#_1\tilde P^3=\nu\;.
\end{split}\]
Since the quantity $\int_{(\R^3)^2} c\,dP$ to be minimized in $C$ is linear with respect to $P$ it is enough to estimate it for each of the components of $\tilde P$ above and using the facts that
$$\begin{cases}
\dist\big(\spt(\nu^1),\spt(\nu^2)\big)>0,\\
0\le\min\{\theta_\delta(x),\theta_\delta(y)\}\le1,\\
\dist\big(\spt\nu,\{X_1,\dots,X_M\}\big)\ge3\delta,
\end{cases}$$
we obtain 
$$\int_{(\R^3)^2} c(x,y)\,d\tilde P (x,y) \le \int_{(\R^3)^2} c(x,y)\,dP (x,y)+\frac{1}{\delta}+\frac{1}{\delta}+ \frac{1}{\dist(\spt(\nu^1),\spt( \nu^2))} $$
which is the desired inequality.
\end{proof}

We are now in position to prove both the existence of the $\Gamma$-limit of the functionals $G_\eps$ and a property of it that will be very useful in the following to obtain an explicit representation formula.

\begin{theo}\label{rhosharp}
For every $\rho\in\PP^-$ the $\Gamma$-limit $G$ of the functionals $G_\eps$ exists and we have
$$G(\rho)=G(\rho^\#).$$
\end{theo}

\begin{proof}
By the compactness of the $\Gamma$-convergence, the $\Gamma$-limit $G$ exists, at least for a subsequence $\eps_n\to0$; as stated in Section \ref{sresca}, since later we will characterize this $\Gamma$-limit explicitly, we may assume it does not depend on the subsequence, so that the entire family $G_\eps$ $\Gamma$ converges and $G^- = G^+ = G$. Let $\rho\in\PP^-$. Writing $\rho=\rho^\#+\rhop$, it is then enough to show the inequalities below:
\be\label{gamma+-}
\begin{split}
&G^-(\rho^\#+\rhop)\le G^+(\rho^\#)\;,\\
&G^+(\rho^\#+\rhop)\ge G^-(\rho^\#)\;.
\end{split}\ee
Let us prove the first inequality in \eqref{gamma+-} for $\rho \in \PP$ in the special case that $\rhop= \rhop_1+\rhop_2$ with $\|\rhop_1\|=\|\rhop_2\|$, $\rhop_i$ smooth, with disjoint and compact supports; in addition we assume that $\dist(\spt\rho^\#,\spt\rhop)>0$. Denote by $\rho_\eps$ a family weakly* converging to $\rho^\#$ and such that
$$G^+(\rho^\#)=\limsup_{\eps\to0} G_\eps(\rho_\eps).$$
Define, for $\delta$ small enough and $\theta_\delta$ as in Lemma \ref{localize}, $\tilde\rho_\eps=\theta_\delta\rho_\eps+\nu_\eps$ where we choose $\nu_\eps= a_\eps \rhop=a_\eps \rhop_1+a_\eps \rhop_2$ and $a_\eps$ is such $\tilde \rho_\eps \in \PP$. Then the assumptions of Lemma \ref{localiz} are satisfied. Then $\tilde\rho_\eps\weak\theta_\delta\rho^\#+\rhop$ and we have by Lemma \ref{localiz}
$$G_\eps(\tilde\rho_\eps)\le G_\eps(\rho_\eps)+\eps K_\delta,$$
where the constant $K_\delta$ only depends on $\delta$.
Passing to the limit as $\eps\to0$ we have
$$G^-(\theta_\delta\rho^\#+\rhop)\le G^+(\rho^\#).$$
Using now the lower semicontinuity of $G^-$ we have, as $\delta\to0$
$$G^-(\rho^\#+\rhop)\le G^+(\rho^\#).$$
In order to extend the inequality above to the general case of $\rho \in \PP^-$ with general $\rhop$ we use again the lower semicontinuity property of $G^-$ and a density argument.

To prove the second inequality in \eqref{gamma+-} we argue in a similar way: take as $\rho_\eps$ a family weakly* converging to $\rho=\rho^\#+\rhop$ and such that
$$G^+(\rho^\# + \rhop)=\limsup_{\eps\to0} G_\eps(\rho_\eps).$$
Thanks to Lemma \ref{localiz}, we may construct $\tilde\rho_\eps=\theta_\delta\rho_\eps+\nu_\eps$ as above, with the cut-off function $\theta_\delta$ as in \eqref{thetad} and taking, for example,
$$\nu_\eps(x)= a_\eps\left[h(x-\frac{x_0}{\eps})+h(x+\frac{x_0}{\eps})\right],$$
with $h \geq 0$ smooth and compactly supported, $x_0 \neq 0$ and $a_\eps$ suitably chosen so
that $\tilde \rho_\eps \in \Pcal$. We get $\nu_\eps\weak0$ and 
$$G_\eps(\tilde\rho_\eps)\le G_\eps(\rho_\eps)+\eps K_\delta.$$
Then, passing to the limit as $\eps\to0$, we have
$$G^+(\rho^\# + \rhop)\ge G^-(\theta_\delta\rho);$$
passing now to the limit as $\delta\to0$ gives $G^+(\rho^\#+ \rhop)\ge G^-(\rho^\#)$ as required.
\end{proof}

By Theorem \ref{rhosharp} all $\Gamma$-limits of subsequences $G_{\eps_n}$ depend only on $\rho^\#$. In Theorem \ref{explform} we will characterize by an explicit formula the $\Gamma$-limit of $G_{\eps_n}$ independently of the subsequence $\eps_n$, obtaining in this way the $\Gamma$-limit of the whole family $G_\eps$.

The following definition of partial or fractional transport cost appeared in the $w^*$ relaxation $\Cb$ of $C$ in equation \eqref{Cpart} and will appear in the formula for the $\Gamma$-limit:
$$C(\rho,2\alpha-1)=
\begin{cases}
\min\big\{\Ccal(\mu)\ :\ 0\le\mu\le\rho,\ \|\mu\|=2\alpha-1\big\}&\mbox{if }\alpha\ge1/2,\\
0&\mbox{otherwise.}
\end{cases}
$$

\begin{defi}
For all $b,Z>0$ and $0\le\alpha\le1$ we define
\be\label{gb}
g_b(Z,\alpha):=\inf_{\|\rho\|=\alpha}\left\{\Tcal(\rho)+bC(\rho,2\alpha-1)-Z\,\Ucal_0(\rho)\right\}\,.
\ee
\end{defi}

\begin{rema} Using again the different homogeneities of the three addenda with respect to the rescaling of measures we have
$$g_b(Z,\alpha) =-\frac14\sup_{\|\rho\|=\alpha}\frac{\big(Z\,\Ucal_0(\rho)-C(\rho,2\alpha-1)\big)_+^2}{\Tcal(\rho)}$$
so that for $\alpha\le1/2$ the equality
$$g_b(Z,\alpha)=-\frac14\alpha Z^2$$
holds. Moreover from \eqref{gb}, again rescaling the measures, for every $\eps>0$ we have
\[ %\label{fundfrac}
\eps^2\Tcal(\rho)+\eps C(\rho,2\alpha-1)-\eps Z\,\Ucal_0(\rho)\ge g_b(Z,\alpha)
\]
for all $\rho\in\PP^-$ with $\|\rho\|=\alpha$. 
\end{rema}

It is clear from the definition of $g_b$ that it is concave non-increasing in $Z$, and we shall prove in Lemma \ref{gblemm} (using the other equivalent Definition \ref{gbgeneral}) that it is convex non-increasing in $\alpha$. Discussing the existence of minimizers is out of the scope of this paper, however we will need some almost optimal measures which we study in the next proposition.
%In the next lemma we will use a cut-off function $\theta \in \C^\infty_c (\R^3)$ 
%such that $0 \leq \theta \leq 1$ supported on a ball $B(0,r)$ with $r$ as big as needed and such that 
%$$\frac{|\nabla\theta|^2}{\theta}\le\eps,\qquad|\Delta\theta|\le\eps.$$

\begin{prop}\label{almostoptimal} For all $\lambda >0$ there exist $r=r(\lambda)$ and $\rho \in \PP^-$ with $\|\rho\|=\alpha$ such that $\spt(\rho)\subset B(0, r)$ and 
$$\Tcal(\rho)+bC(\rho,2\alpha-1)-Z\, \Ucal_0(\rho)\le g_b(Z,\alpha)+ \lambda.$$
\end{prop}
\begin{proof} This is an indirect variant of Lemma \ref{localiz} in which the fractional transport cost appears and $\eps=1$.
	%a slightly different control of the kinetic and potential energy is needed.
So we carefully apply Lemmas \ref{localize} and \ref{localiz}. 
We start from $\tilde \rho \in \PP^-$ such that $\|\tilde \rho\|=\alpha$ and
$$\Tcal(\tilde \rho)+bC(\tilde \rho, 2\alpha-1) -Z\, \Ucal_0(\tilde \rho) \le g_b(Z,\alpha)+ \frac{\lambda}{3}.$$
We need to modify $\tilde \rho$ so that the support becomes compact. Let also $\tilde\mu\le\tilde\rho$ with $\|\tilde\mu\|=2\alpha-1$ and $\Ccal(\tilde\mu)=C(\tilde\rho,2\alpha-1)$. Consider $\theta_\delta$ and $\nu$ as in Lemma \ref{localiz} with $\delta$ large enough so that we may have $K$ small and
$$
2K+\frac{Z+2b}{\delta} +\frac{b}{\dist(\spt(\nu_\eps^1),\spt( \nu_\eps^2))} \leq \frac{\lambda}{3}.
$$
Here we set $\rho = \theta_\delta \tilde \rho+\nu$ and assume $\| \rho\|=\alpha$ (instead of $\rho \in \PP$ as in Lemma \ref{localiz}). Then
$$\int\nu\,dx=\alpha-\|\theta_\delta\tilde\rho\|.$$
By Lemma $\ref{localize}$ with $\eps=1$ we have
$$\Tcal( \rho ) - Z \Ucal_0 (\rho)\le\Tcal(\tilde \rho)- Z\Ucal_0 (\tilde \rho) + \frac{\lambda}{3}\;.$$
To estimate the fractional transport term $C(\rho, 2 \alpha -1)$ we consider $\mu= \theta_\delta \tilde \mu + \beta_0 \nu$ where $\beta_0$ is such that the total variation of $\mu$ is equal to $2\alpha-1$. To show that $\mu \leq \rho$ we need to show that $\beta_0 \leq 1$. This is equivalent to say that 
$$ 2 \alpha -1 -\int\theta_\delta d\tilde\mu \leq \int d\nu =\alpha - \int \theta_\delta d\tilde\rho$$ 
which is the inequality 
$$ \int \theta_\delta d (\tilde\rho-\tilde \mu) \leq 1-\alpha = \int d(\tilde \rho-\tilde \mu).$$
We then apply the transport estimate of Lemma \ref{localiz}, up to a rescale of the measures, to get
$$C(\rho, 2\alpha-1) \leq \Ccal (\mu) \leq \Ccal (\tilde \mu) +\frac{\lambda}{3} = C(\tilde\rho, 2\alpha-1)+ \frac{\lambda}{3}.$$
which concludes the proof.
\end{proof}

%\begin{coro}\label{coroll1}
%For every $\lambda>0$, there exists $\rho_\eps \rightharpoonup \alpha \delta_0$ such that 
%$$ \lim_{\eps \to 0}\eps^2T(\rho_\eps)+\eps bC(\rho_\eps, 2\alpha-1) -\eps U_Z(\rho_\eps) \leq -\frac14 m_b(Z, \alpha)+ \lambda$$
%and $\spt (\rho_\eps) \subset B(0, \delta)$ for $\eps$ small enough. 
%\end{coro}
%\begin{proof} Let $\rho$ be as in Proposition \ref{almostoptimal} above. And let $\rho_\eps (x):= \frac{1}{\eps^3} \rho(\frac{x}{\eps})$. 
%Since 
%$$ \eps^2T(\rho_\eps)+\eps bC(\rho_\eps, 2\alpha-1) -\eps U_Z(\rho_\eps) = T(\rho)+ bC(\rho, 2\alpha-1) - U_Z(\rho)$$ 
%this generalized sequence satisfy all the requirements. 
%\end{proof}

A short investigation of the structure of certain optimal transport plans will be used in the next theorem.

\begin{lemm}\label{staylocal} Let $X_0,\dots,X_M\in\R^d$ and let $\delta>0$ be such that $\delta\ll\min_{i,j}|X_i- X_j|$. Consider $\rho= \sum_{i=0} ^{M} \alpha_i \rho_i$ with $\rho_i$ probability measures such that $\spt \rho_i \subset B(X_i, \delta)$ and $\sum_{i=0}^M \alpha_i=1$. Let $P$ be an optimal transport plan for $\rho$, then 
\[
\begin{cases}
P(B(X_i, \delta) \times B(X_i, \delta))=0 &\text{if }\alpha_i \le1/2\,,\\
P(B(X_i, \delta) \times B(X_i, \delta))=2 \alpha_i -1 &\text{if }\alpha_i >1/2\,.
\end{cases}
\]
It follows that there exists $K$ depending only on $\min_{i \neq j}{|X_i-X_j|-2\delta}$ such that the following alternative holds
\begin{enumerate}
\item \label{stimup} $C(\rho)\le K$ if $\alpha_i \le1/2$ for all $i$,
\item \label{stimboth} if $\alpha_i>1/2$ for some $i$ then 
$$C(\alpha_i \rho_i, 2\alpha_i-1) \leq C(\rho) \leq C(\alpha_i \rho_i, 2\alpha_i-1)+K .$$
\end{enumerate}
\end{lemm}

Note that in the above statement we add a point $X_0$, not corresponding to a nucleus, to the points $X_1,\ldots,X_M$: this will be handy in the proof of Theorem \ref{explform} below.

\begin{proof} Let $i\in \{0, \dots, M\}$ and let $l_1:=P(B(X_i,\delta)\times B(X_i,\delta))$, we may compute
\[\begin{split}
P((B(X_i,\delta)\times\R^d)\cup(\R^d\times B(X_i,\delta)))
&=P(B(X_i,\delta)\times\R^d)+P(\R^d\times B(X_i,\delta))\\
&\quad-P(B(X_i,\delta)\times B(X_i,\delta))\le2\alpha_i-l_1.
\end{split}\]
For the case $\alpha_i\le1/2$, if $l_1>0$ then, since $2\alpha_i-l_1<1$, there exist $j,k\neq i$ such that 
$$ l_2 := P (B(X_j,\delta) \times B(X_k, \delta)) >0. $$
Define $s=\min\{l_1,l_2\}$ and
$$P_1=\frac{s}{l_1}P_{|B(X_i,\delta)\times B(X_i,\delta) },\qquad P_2=\frac{s}{l_2}P_{|B(X_j,\delta)\times B(X_k,\delta) },$$
and rewrite $P=P_1+P_2+P_R$ where $P_R$ is defined by this same equality. Since the quantity $c(P)=\int_{(\R^3)^2} c\,dP$ is linear in $P$ we have
$$c(P)= c(P_1)+c(P_2)+c(P_R)\ge\frac{s}{2\delta}+c(P_R).$$
Define $\tilde P$ by
$$\tilde P=\pi_1^\# P_1\otimes\pi_2^\# P_2+\pi_1^\# P_2\otimes\pi_2^\# P_1+P_R\,,$$
so that it has the same marginal as $P$. Concerning the transportation cost we have
$$c(\tilde P) \leq \frac{s}{|X_i-X_k| -2 \delta}+\frac{s}{|X_j-X_i| -2 \delta} + c(P_R)$$ 
which is smaller than $c(P)$ for $\delta$ as in the assumptions, and this contradicts the optimality of $P$.

Analogously, if $\alpha_i>1/2$, since $\rho\big(\R^d\setminus B(X_i,\delta)\big)=1-\alpha_i$ we have
\[\begin{split}
l_1&=P\big(B(X_i,\delta)\times B(X_i,\delta)\big)\\
&=P\big(B(X_i,\delta)\times\R^d\big)-P\big(B(X_i,\delta)\times(\R^d\setminus B(X_i,\delta)\big)\\
&\ge\alpha_i-(1-\alpha_i)=2\alpha_i-1.
\end{split}\]
The strict inequality would imply again that $2 \alpha_i-l_1 <1$ and then again there exist $j,k\neq i$ such that 
$ P (B(X_j,\delta) \times B(X_k, \delta)) >0.$ This would contradict the optimality of $P$ as in the first case.

To deduce \eqref{stimup} we consider an optimal transport plan $P_{opt}$ and remark that for $x$ in the support of $P_{opt}$ for all $i,j$
$$\frac{1}{|x_i-x_j|} \leq \frac{1}{ \min_{i \neq j}{|X_i-X_j|-2\delta}}$$
so that we can take
$$K=\frac{N(N-1)}{2\min_{i\ne j}{|X_i-X_j|-2\delta}}\;.$$
To prove the first inequality in \eqref{stimboth} consider a symmetric optimal transport plan $P$ for $C(\rho)$. Let $P'=P\lfloor B(X_i,\delta)\times B(X_i,\delta)$ and let $\mu:= \pi_1^\# P' (= \pi_2^\# P')$. Clearly $0 \leq \mu \leq \alpha_i\rho_i$ and $\|\mu\|= \|P'\|= 2 \alpha_i-1$,
then 
\[C(\alpha_i\rho_i, 2 \alpha_i-1) \leq c(P') \leq c(P')+c(P-P')= C(\rho).\]
For the second inequality in \eqref{stimboth}, let $\mu \leq \alpha_i\rho_i$ with $\|\mu\|=2\alpha_i -1$ be such that 
\[\Ccal(\mu)= C(\alpha_i\rho_i, 2 \alpha_i-1),\]
and let $P_\mu$ be an optimal plan for $\mu$. Consider an optimal plan $P_R$ for $\rho-\mu$. The plan $P=P_\mu+P_R \in \Pi(\rho)$ and then
$$C(\rho)\le c(P)=c(P_\mu)+c(P_R)=\Ccal(\mu)+\Ccal(\rho-\mu)=C(\alpha_i\rho_i,2\alpha_i-1)+\Ccal(\rho-\mu).$$
We conclude observing that
$$\Ccal(\rho-\mu)=\Norm{\rho-\mu}C\left(\frac{\rho-\mu}{\Norm{\rho-\mu}}\right)\le K\Norm{\rho-\mu}\le K$$ since
$$\rho-\mu= \alpha_0 \rho_0+\dots+(1-\alpha_i)\rho_i+\dots+\alpha_M \rho_M$$
and
$$\alpha_0,\dots,1-\alpha_i,\dots,\alpha_M\le\frac{1}{2}<\alpha_i=\frac{\|\rho-\mu\|}{2}$$
so \eqref{stimup} applies to $\frac{\rho-\mu}{\Norm{\rho-\mu}}$.
\end{proof}

By Theorem \ref{rhosharp} we may now focus on the formula for $G(\rho^\#)$.

\begin{theo}\label{explform}
Let $\rho\in\PP^-$ be such that
$$\rho^\#=\sum_{i=1}^M \alpha_i\delta_{X_i}\qquad\hbox{with $\alpha_i\ge0$ and }\sum_i \alpha_i\le1.$$
Then the following formula holds:
$$G(\rho)=\sum_i g_b(Z_i,\alpha_i).$$
\end{theo}

\begin{proof}
It follows from Theorem \ref{rhosharp} that $G(\rho)=G(\rho^\#)$, so we shall prove the result for $\rho=\rho^\#$.
	
Let $\delta>0$ be such that $\delta\ll\min_{i,j}|X_i-X_j|$. We start with the $\Gamma-\limsup$ inequality. For every $\eps>0$ we add to the $M$-uple $(X_1, \dots, X_M)$ an additional point $Y_\eps$ such that
$$\lim_{\eps\to0}\|Y_\eps\| \to \infty. $$
Let $\lambda>0$, let $\rho_i$ be a measure with compact support obtained from Proposition \ref{almostoptimal} applied with parameters 
$\lambda_i=\frac{\lambda}{M},\alpha_i,Z_i$. Let $h\in\C^\infty_0\big(B(0,\delta)\big)$ a positive function such that
$$\int h\,dx=1\qquad\hbox{and}\qquad\int\frac{|\nabla h|^2}{h}\,dx<+\infty$$
and define
$$\rho_\eps^0(x)=\alpha_0 h(x-Y_\eps)\qquad\hbox{with }\alpha_0=1-\sum_{i=1}^M \alpha_i\;.$$
For $i=1,\dots,M$ let
$$\rho_\eps^i(x)=\rho_i^{1/\eps}(x-Xi) = \frac{1}{\eps^3}\rho_i \left(\frac{x-X_i}{\eps}\right)\qquad\hbox{and}\qquad\rho_\eps(x)=\sum_{i=0}^M\rho_\eps^i(x)\;.$$
For $\eps$ small enough the supports of $\rho_\eps^0$ and $\rho_\eps^i$ are contained in $B(Y_\eps, \delta)$ and $B(X_i, \delta)$ respectively. 
We estimate $G_\eps (\rho_\eps)$ from above. Since
\[\begin{split}
&T(\rho_\eps)=\sum_{i=1}^M \Tcal(\rho_\eps^i)+\alpha_0 \Tcal(h),\\
&U(\rho_\eps)\ge\sum_{i=1}^M Z_i \, \Ucal_{X_i}(\rho^i_\eps),
\end{split}\]
we deduce
\be\label{sommario1}
G_\eps (\rho_\eps) \leq \sum_{i=1}^M (\eps ^2 \Tcal(\rho_\eps^i)-\eps Z_i\, \Ucal_{X_i} (\rho^i_\eps))+\eps^2\alpha_0 \Tcal(h) + \eps C(\rho_\eps).
\ee
We then need to decompose $C(\rho_\eps)$.
By Lemma \ref{staylocal} if $\alpha_i\le1/2$ for all $i$ then $C(\rho_\eps) \leq K$ and passing to the $\limsup$
$$ G^+ (\rho^\#)\le\limsup_{\eps \to 0} G_\eps (\rho_\eps)\leq \sum_i g_b(Z_i,\alpha_i) + \lambda $$
because the last two terms in (\ref{sommario1}) go to $0$, and the first term is computed by the homogeneity of the energy and the choice of $\rho_i$.
If for one $i\in \{1,\dots, M\}$, $\alpha_i>1/2$, we assume without loss of generality that it is $\alpha_1$, then by Lemma \ref{staylocal}
$$C(\rho_\eps)\le C(\rho_\eps^1, 2\alpha_1-1) +K.$$ 
Then 
\[\begin{split}
G_\eps(\rho_\eps)&\le\eps^2 T(\rho_\eps^1)-\eps U_{Z_i}(\rho^1_\eps)+\eps C_2(\rho_\eps^1,2\alpha_1-1)\\
&\qquad+\sum_{i=2}^M (\eps^2 T(\rho_\eps^i)-\eps U_{Z_i}(\rho^i_\eps))+\eps^2\alpha_0 T(h)+\eps K
\end{split}\]
and again we conclude by the homogeneity of the energy and the choice of $\rho_i$. The case $\alpha_0>1/2$ can be excluded by considering a second sequence $\tilde Y_\eps=-Y_\eps$ and then defining
$$\rho^0_\eps=\frac{\alpha_0}{2}h(x-Y_\eps)+\frac{\alpha_0}{2}h(x-\tilde Y_\eps),$$
which has the same properties needed in the proof but do not concentrate too much mass in a ball of radius $\delta$: we are then applying Lemma \ref{staylocal} with the $M+2$ points $X_1,\ldots,X_M,Y_\eps,\tilde Y_\eps$.

We now deal with the $\Gamma-\liminf$ inequality. Let $\rho_\eps\weak\rho^\#$. By Lemma \ref{localiz} we can replace $\rho_\eps $ by $\theta_\delta\rho_\eps+\nu_\eps$ with $\nu_eps$ chosen as in Example \ref{exthetanu} so that the constant $K$ does not depend on $\delta$. Since 
$$\liminf_{\eps \to 0} G_\eps(\theta_\delta\rho_\eps+\nu_\eps) \leq \liminf_{\eps \to 0} G_\eps(\rho_\eps).$$
 We denote by $\rho_\eps^i := {\theta_\delta \rho_\eps} _{| B(X_i, 2 \delta)}$
%and we remark that if $\alpha_1$ (or any of the $\alpha_i$'s up to a reordering of the indexes) satisfies $\alpha_1>1/2$ then for $\eps$ small enough $\rho_\eps^i\big(B(X_1,2\delta)\big)>1/2.$
and we have, for some constant $K_\delta$ that does not depend on $\eps$, the inequality
$$G_\eps(\rho_\eps)\ge\sum_{i=1}^M(\eps^2 T(\rho_\eps^i)-\eps U_{Z_i}(\rho^i_\eps))+\eps^2\alpha_0 T(\nu_\eps)+\eps C(\rho_\eps)-\eps\frac{K}{2\delta}.$$
Again we need to look at $C(\rho_\eps)$. If $\alpha_i\leq 1/2$ for all $i$ we just use that $C(\rho_\eps)\geq0$ and get
$$G_\eps(\rho_\eps)\ge-\frac14\sum_{i=1}^M Z_i^2\|\rho_\eps^i\|+\eps^2\alpha_0 T(\nu_\eps)-\eps\frac{K}{2\delta}\,.$$
When $\eps \to 0$ and then $\delta \to 0$ we obtain 
\[\liminf_{\eps \to 0} G_\eps(\rho_\eps)\ge-\frac14\sum_{i=1}^M Z_i^2 \int _{B(X_i,2\delta)} \rho dx \stackrel{\delta \to 0}{\longrightarrow}-\frac14\sum_{i=1}^M Z_i^2 \alpha_i = \sum_i g_b(Z_i,\alpha_i)\]
which concludes the proof in this case.

If $\alpha_1>1/2$ (or any $\alpha_i$ up to reindexing) then for $\eps$ small enough $\|\rho^\eps_i\|=\rho_\eps^1\big(B(X_1,2\delta)\big)>1/2$ and then by Lemma \ref{staylocal}
\[\begin{split}
G_\eps (\rho_\eps)&\ge\eps^2 T(\rho_\eps^1)+\eps C(\rho_\eps^1,2\|\rho_\eps^1\|-1)-\eps U_{Z_1}(\rho^1_\eps)\\
&\qquad+\sum_{i=2}^M(\eps^2 T(\rho_\eps^i)-\eps U_{Z_i}(\rho^i_\eps))+\eps^2\alpha_0 T(\nu_\eps) +\eps\frac{K}{2\delta}.
\end{split}\]
By the homogeneity of the three terms 
 \[
 G_\eps (\rho_\eps) \ge g_b(Z_1,\|\rho^1_\eps\|)-\frac14\sum_{i=2}^M Z_i^2\|\rho_\eps^i\|+\eps^2\alpha_0 T(\nu_\eps) +\eps\frac{K}{2\delta}.
 \]
Passing again to the limit for $\eps \to 0$ and then $\delta \to 0$, and using the lower semi-continuity of $g_b(Z_i,\cdot)$ allows to conclude in this case.
\end{proof}

\begin{rema} Since $\sum_i \alpha_i \leq 1$ at most one of the $\alpha_i$ may be greater then $1/2$ and we will always assume that is the first $\alpha_1$.
Then, recalling that by definition,
$$\rho=\rho^\#+\rhop=\sum_{i=1}^M\alpha_i\delta_{X_i}+\rhop\qquad\forall\rho\in \PP^-,$$
we have two possible cases:
$$G(\rho)=G(\rho^\#)=-\frac{1}{4}\sum_{i=1}^M\alpha_i Z_i^2,\qquad\mbox{if }\alpha_1\le\frac{1}{2},$$
or
$$G(\rho)=G(\rho^\#)=g_b(Z_1,\alpha_1)-\frac{1}{4}\sum_{i=2}^M\alpha_i Z_i^2,\qquad\mbox{if }\alpha_1>\frac{1}{2}.$$
\end{rema}

We are now in a position to study the minimization problem
\be\label{minpbrho}
\min\big\{G(\rho)\ :\ \rho\in\PP^-\big\}.
\ee

\begin{theo}\label{solprob}
Let us assume that $M\ge 2$. Then the minimization problem \eqref{minpbrho} has a solution $\rho\in\PP^-$. Moreover, every such a minimizer $\rho$ belongs to $\PP$ and is of the form
$$\rho=\rho^\#=\sum_{i=1}^M\alpha_i\delta_{X_i}\qquad\hbox{with }\sum_{i=1}^M\alpha_i=1.$$
\end{theo}

Note that the case $M=1$ is discussed in Remark \ref{solM1} below.

\begin{proof}The existence of an optimal $ \rho$ follows from the weak$^*$ compactness of $\PP^-$ and lower semicontinuity of the $\Gamma$-limit $G$. For such a $\rho$ , set $\a_i= \rho(\{X_i\})$. Then, by Theorem \ref{explform}, $\rho^\#:=\sum_{i=1}^M\alpha_i\delta_{X_i}$ is also optimal and we claim that $\sum_{i=1}^M \a_i =1$. Indeed if $\sum_{i=1}^M \a_i <1$, then there exists $j$ such that $\a_j<1/2$ and we may consider $\bar\rho := \rho + \eta \delta_{X_j}$ where $\eta$ is such that
$$0<\eta<\min\left\{\frac12 -\a_j, 1-\sum \a_i\right\}.$$
Then, by applying again Theorem \ref{explform}, we obtain
\begin{align*}
G(\bar\rho)=G(\rho^\#+\eta\delta_{X_j})&=-\frac{(\a_j +\eta)}{4} Z_j^2 + \sum_{i\not=j}g_b(Z_i,\alpha_i) \\
& =-\frac1{4}\eta Z_j^2 + G(\rho^\#) \;< \;G(\rho^\#)
\end{align*}
where we used the fact that $\a_j+\eta<1/2$. 
\end{proof}

\begin{rema}\label{Gmin}
To illustrate the previous analysis, we now discuss the structure of the minimizers of problem \eqref{minpbrho} in the special case $N=M=2$. Without loss of generality we may assume that $Z_1 \geq Z_ 2$, then it follows from Theorem \ref{solprob} that the minimizers of \eqref{minpbrho} are of the form
$$\rho=\alpha\delta_{X_1}+(1-\alpha)\delta_{X_2}$$
with $\alpha\in[0,1]$ minimizing the problem
\be\label{minGmin}
\min\big\{g_b(Z_1,\alpha)+g_b(Z_2,1-\alpha)\ :\ \alpha\in[0,1]\big\}\,.
\ee
From Proposition \ref{gbprop} and Remark \ref{gbpropN2}, we know that the convex non-increasing functions $\beta \mapsto g_b(Z_i,\beta)$ satisfy
$$\begin{cases}
\ds g_b(Z_i,\beta)=-\frac{Z_i^2}{4}\,\beta&\quad\text{for }\beta\in\left[0,\frac{1}{2}\right]\,,\\
\ds g_b(Z_i,\beta)>-\frac{Z_i^2}{4}\,\beta&\quad\text{for }\beta\in\left]\frac{1}{2},1\right]\,.
\end{cases}$$
As a consequence for $Z_1=Z_2$ we obtain that the minimum is uniquely attained for $\alpha=\frac{1}{2}$ : we thus recover a more precise result than in Remark \ref{G0min} where in that case any $\alpha \in [0,1]$ would lead to a solution. Here, due to the correlation term $C$, each nucleus gets exactly one electron (see also Remark \ref{friesecke} below). \\
We now turn to the case $Z_1 > Z_2$. In that case it holds $g_b(Z_1,\beta) \leq g_b(Z_2,\beta)$ for all $\beta$.
Moreover from the above properties of $g_b$ it follows that the minimum in \eqref{minGmin} is equal to
\[
\min \left\{ g_b(Z_1,\alpha) -\frac{Z_2^2}{4}\, (1-\alpha) : \alpha \in \left[ \frac{1}{2},1\right] \right\}\,.
\]
Now, since $g_b(Z_1,1) > - \frac{Z_ 1^2}{4}$ we may choose $Z_2$ close enough to $Z_1$ (yet keeping $Z_2<Z_1)$ such that
\[
g_b(Z_1,1) > - \frac{Z_ 1^2+Z_2^2}{8}\,,
\]
in which case the minimum in \eqref{minGmin} is attained for some $\alpha < 1$ : this is quite different from 
what obtained in Remark \ref{G0min} where the minimum would be only for $\alpha=1$.
Somehow this allows for a continuity of the solution set of \eqref{minGmin} as $Z_2$ gets closer to $Z_1$.

Since $\rho$ represents the probability distribution of the $N$ electrons, the presence of values of $\alpha_i$ which are not of the form $k/N$ should be interpreted as the presence of a fractional number of electrons. This fact already appeared in the literature (see for instance \cite{msgg13, pplb82}) and has a reasonable interpretation in terms of time-averaging.
\end{rema}

\begin{rema}\label{friesecke}
We now discuss more extensively the $H_2$ molecule bond dissociation, and we show how our results compare to the results of \cite[Theorem 5.1]{cfm14}. The $H_2$ molecule corresponds to $M=2$ nuclei with charges $Z_1=Z_2=1$ and $N=2$ electrons. The physical total energy for this molecule when the nuclei are located at $X_i/\eps$ is given by
\be\label{h2}\frac{2\eps}{|X_1-X_2|}+ \inf \left\{2 T(\rho) + C(\rho) - 2 (U_{X_1/\eps} (\rho) + U_{X_2/\eps} (\rho))\ :\ \|\rho\|=1 \right\}
\ee
where the first term in the nucleus-nucleus Coulomb interaction and is vanishing with 
$\eps$. The $\inf$ part corresponds to $2\inf G_\eps$ for $b=1/2$. 
The representation of the $\Gamma$-limit of $G_\eps$ in this case is given by
$$G(\rho)= g_b(1,\rho(\{X_1\}))+ g_b(1,\rho(\{X_2\}))$$
and, according to Remark \ref{Gmin} above, the minimum of $G$ is attained for $\rho(\{X_1\})=\rho(\{X_2\})=\frac{1}{2}.$ 
It follows that the energy \eqref{h2} above converges, as $\eps\to 0$ to
$$4\,g_b\left(1,\frac{1}{2}\right)=4\min\left\{\Tcal(\rho)-\Ucal_0(\rho)\ :\ \|\rho\|=\frac{1}{2}\right\}=2\min\left\{T(\rho)-U_0(\rho)\ :\ \|\rho\|=1\right\}$$
which is twice the energy of an hydrogen atom as proved in \cite[Theorem 5.1]{cfm14}. Note that with our normalization $\|\rho\|=\frac{1}{2}$ represents the energy of one electron when $N=2$.
\end{rema}

%%%%%%%%%%%%%%%%%%%%%%%%%%%%%%%%%
\section{The general interacting case}%\label{gen-interact}
%%%%%%%%%%%%%%%%%%%%%%%%%%%%%%%%%
In this section we consider the general case $b>0$ and $N\ge3$, for which
the proof of existence as well as a full characterization of the $\Gamma$-limit $G$ of the family of functionals $\{G_\eps\}_\eps$ seems a hard issue.
%Hereafter we give some partial attempts to identify the $\Gamma$-limit functional in some particular cases.
In view of the successive steps of the preceding section \S\ref{two-interact} for the case $N=2$,
we expect to have the following properties on $G$ over $\PP^-$:
\begin{enumerate}[label=(P\arabic*)]
\item\label{P1} $G(\rho)$ only depends on the restriction of $\rho$ to the nuclei $X_1,\ldots,X_M$ in the sense
$$G(\rho) = G(\rho^\# + \rhop) = G(\rho^\#)$$
where we use the decomposition $\rho=\rho^\#+\rhop$ of Section \ref{two-interact} with
$$\rho^\#=\sum_{i=1}^M\rho(\{X_i\})\delta_{X_i}\,;$$
\item\label{P2} there exists a function $g_b$ given as a generalization to any $N \geq 2$ of the Definition \ref{gb} such that
$$G(\rho^\#)=\sum_{i=1}^M g_b(Z_i,\rho(\{X_i\}))\,.$$
\end{enumerate}
Note that in the case $N=2$ of Section \ref{two-interact}, \ref{P1} is obtained in Theorem \ref{rhosharp} via the localization Lemma \ref{localiz} which allows to control the transport term $C$. On the other hand, the proof of \ref{P2} relies mainly on Lemma \ref{staylocal} which is linked to a deep understanding of the lower semicontinuous envelope $\Cb$ of $C$, that is only fully characterized for this particular value of $N$.

In Subsection \ref{smge1} below, we do obtain the properties \ref{P1} and \ref{P2} for the general case under some regularity assumptions on the subprobability $\rho$, which in particular is required to give small mass to the nuclei $X_i$ (precisely $\rho(\{X_i\})\le1/N$ for all $i$). Then in Subsection \ref{ssm1} we derive \ref{P2} for the special case $M=1$, i.e. when there is only one nucleus.

\subsection{Full characterization of $G$ for particular subprobabilities}\label{smge1}

We consider the general case $b>0$, $M\ge 1$ and $N \geq 2$.

We represent every measure $\rho\in\PP^-$ as
\[
\rho=\sum_{i=1}^M\alpha_i\delta_{X_i}+\rhop
\]
with $\rhop$ which does not charge the points $X_i$.
%We give a representation result for the $\Gamma$-limit functional $G(\rho)$, which allows us to prove that the optimal measures $\rho$ for $G$ are actually probabilities, that is no mass is lost at infinity.
In Proposition \ref{l2} we already showed that whenever $\alpha_1=\dots=\alpha_M=0$ then $G(\rho)=0$. In this section we compute $G$ in the case $\alpha_i\le1/N$ for all $i$. A particular case is $M=N$ and $\alpha_i=1/N$ for all $i$ which is the expected optimal configuration for the hydrogen molecule $H_N$.
%\red{THIERRY : il faut retravailler cette phrase ! N'a plus de sens maintenant : In this last case we extend and give an interpretation on the results of \cite{cfm14} for $H_2$ (see Remark \ref{friesecke} below).}

%We also compute $G$ in the case $\alpha_i=1$ for some $i$ and $\alpha_j=0$ whenever $i\neq j$. The next step would be to consider the case $\rho= \sum_{i=1}^M \alpha_i \delta_{X_i}$ with $\sum_{i=1}^M \alpha_i=1$ (i.e. $\rhop=0$).
%This case seems to be more involved so that, in this paper, we will limit our efforts to the case $M=2$ and $N=3$ which already presents some new aspects. \red{\bf maybe the case $N=M=2$ could be fully treated}
%\blue{thierry : il me semble que cette phrase est ancienne !!}

In the following proof, we shall need the notion of concentration of a finite measure $\rho$,
which is denoted $\mu(\rho)$ and defined by
$$ \mu(\rho):= \lim_{r \to 0} \;\sup_{x} \rho (B(x,r)).$$
In particular $\mu (\alpha \delta_X) = \alpha$ and, in general $\mu(\rho)\geq \alpha$ implies that $\rho= \alpha \delta_X + \sigma$ for some point $X$ and a nonnegative measure $\sigma$. We can now state our result.

\begin{theo}\label{alphaileq1N}
Let $\rho=\alpha_1\delta_{X_1}+\dots+\alpha_M\delta_{X_M}+\rhop$ with $\alpha_i\le1/N$ for all $i$ and $\mu (\rhop)\le1/N$, then
$$G(\rho)= -\sum_{i=1}^M \alpha_i \frac{Z_i^2}{4}.$$
\end{theo}

\begin{proof} Since for every $\rho$
$$G^- (\rho) \geq G^0 (\rho)=-\sum _{i=1}^M \alpha_i \frac{Z_i^2}{4},$$
it is enough to prove that
$$ G^+(\rho)\le-\sum_{i=1}^M \alpha_i \frac{Z_i^2}{4}.$$
We first consider the case $\rho \in \Pcal$. By Theorem \ref{liebG}, there exists a sequence $\rho_\eps\weak\rho$ such that
$$G^0_\eps(\rho_\eps) = \eps^2 T(\rho_\eps) - \eps U(\rho_\eps) \;\longrightarrow\; -\sum_{i=1}^M \alpha_i \frac{Z_i^2}{4}\,.$$
Note that since $\rho \in \Pcal$ the sequence $\{\rho_\eps\}_\eps$ converges narrowly to $\rho$ and $\limsup \mu(\rho_\eps) \leq \mu(\rho)$. Assume first that $\mu(\rho)<1/N$.
%, which implies in particular the strict inequalities $\alpha_i<1/N$.
Then by a result of \cite{ColDiMStr2018} (a proof of this result is also available in \cite{tesiStra} or, in Proposition 2.5 of \cite{bcdp17} for the more restrictive case $\mu(\rho)<\frac{1}{N(N-1)^2}$) the transport cost $C$ is uniformly bounded on the family $\{\rho_\eps\}_\eps$, so we have
$$\lim_{\eps\to0}\eps C(\rho_\eps)=0.$$
We conclude that
$$G^+(\rho)\le\limsup_{\eps\to0}G_\eps (\rho_\eps)=G^0 (\rho)=-\sum_{i=1}^M \alpha_i\frac{Z_i^2}{4}.$$
We obtain the general case by approximation since $G^+$ is lower semicontinuous and the right hand side is upper semicontinuous with respect to weak* convergence.
\end{proof}

%\red{\bf 
%	\begin{rema}\label{friesecke}
%		In the above proof, we have seen that the correlation term $C$ is bounded on the approximating sequences $\{\rho_\eps\}_\eps$ so that when multiplied by $\eps$ it is canceled in the limit. This is basically what was observed in \cite[Theorem 5.1]{cfm14} for the special case of the molecule $H_2$ : (...) TO BE COMPLETED which corresponds here to $\alpha_1=\alpha_2=\frac{1}{2}$ 
%		\end{rema}}
%\red{\bf Thierry : il me semble qu'en fait il faut reprendre cette preuve, parceque le r\'esultat de continuit\'e qu'on a obtenu est un peu diff\'erent, il dit cela :
%\begin{theo}[Theorem 3.11 in \cite{bcdp17}]\label{C-continuous}
%Let $\rho_1,\rho_2 \in \Pcal$ with $\rho_1 - \rho_2 \in L^1$ be such that
%\[
%\exists r >0, \; \forall i \in \{1,2\}, \qquad \sup_{x \in \R^3} \rho_i(B(x,r)) < \frac{1}{N(N-1)^2},
%\]
%then there exists a constant $\kappa$ (depending only on $N$ and $r$) such that
%\[
%|C(\rho_1)-C(\rho_2)| \le \kappa \,\|\rho_1-\rho_2\|_{L^1}.
%\]
%\end{theo}
%Il faut donc faire attention, mais \c{c}a devrait \^etre vrai quand m\^eme...
%}

%%%%%%%%%%%%%%%%%%%%%%%%%%%%%%%%%
\subsection{The special case $M=1$ and $b>0$}\label{ssm1}

In this case, we may assume that the position of the unique nucleus is $X_1 = 0$.
We prove \ref{P2} with the following definition for the function $g_b$.

\begin{defi}\label{defgb}
For $R >0$, $Z \in \R_+$ and $\alpha \in [0,1]$ we define
\[ %\label{gbR}
g_b^R (Z,\alpha):=\inf\left\{T(\rho)+bC(\rho)-Z\,U_0(\rho)\ :\ \rho\in\PP,\ \int_{B(0,R)}d\rho \le\alpha\right\} 
\]
and then set 
\be\label{gbgeneral}
g_b(Z,\alpha):=\sup\big\{g_b^R(Z,\alpha)\ :\ R>0\big\}\,.
\ee
\end{defi}

Note that the definition above could look ambiguous because we already defined the function $g_b$ in \eqref{gb}. However, we shall obtain as a consequence of Theorem \ref{P2M1} that the two definitions of the function $g_b$ given in \eqref{gbgeneral} above and in \eqref{gb} do coincide for the case $N=2$ (see Remark \ref{gb-coherent} below).

Let us state some properties of the fonctions $g_b$ and $g_b^R$ that shall be usefull below.

\begin{lemm}\label{gblemm}
For any $R> 0$, the functions $\alpha \mapsto g_b^R(Z,\alpha)$ is convex, continuous and non-increasing on $[0,1]$ and $g_b^R(Z,\alpha) \geq -\frac{Z^2}{4}\alpha$. The same holds for $\alpha \mapsto g_b(Z,\alpha)$.
\end{lemm}

\begin{proof}
First note that the functions $g_b^R(Z,\cdot)$ and $g_b(Z,\cdot)$ may in fact be defined through the same formulas on $[0,+\infty[\,$, then being constant and equal to $g_b(Z,1)$ on $[1,+\infty[\,$.
Moreover, the functions $g_b^R(Z,\cdot)$ are clearly non-increasing, so that this also holds for $g_b(Z,\cdot)$.
Finally the convexity of $\alpha \mapsto g_b^R(Z,\alpha)$ follows from the convexity of the map
$$\rho\mapsto G_1(\rho)=T(\rho)+bC(\rho)-Z\,U_0(\rho),$$
and taking the supremum over $R$ yields the convexity for $g_b(Z,\cdot)$. It remains to show the continuity on $[0,1]$. To see this we first compute
$$\forall R>0,\ \forall\alpha\in[0,1],\qquad g_b^R(Z,\alpha)\ge\inf_{\rho\in\PP}G_1^0(\rho)
=\inf_{\rho\in\PP}G^0=-\frac{Z^2}{4}\;,$$
where the last equality follows from \eqref{usualtrick}. Consider now $R>0$ and $\rho\in\PP$ smooth and such that $\int_{B(0,R)}d\rho=0$, then for all $\lambda >1$ one has $\int_{B(0,R)}d\rho^{1/\lambda}=0$ and
\[
g_b^R(Z,0) \leq G_1(\rho^{1/\lambda}) \leq \frac{1}{\lambda^2} T(\rho) + \frac{b}{\lambda}C(\rho) - \frac{Z}{\lambda} U_0(\rho)
\]
and letting $\lambda$ go to infinity yields $g_b^R(Z,0)\leq 0$. From the preceding we thus get the continuity of $g_b^R(Z,\cdot)$ and $g_b(Z,\cdot)$ at $\alpha=0$. Finally these convex non-increasing functions take their values in $\left[ -\frac{Z^2}{4}, 0\right]$, so they are bounded and thus continuous on $\,]0,+\infty[\,$.
\end{proof}

We are now in position to prove our main result in this particular case.

\begin{theo}\label{P2M1}
For $M=1$ and $X_1=0$ it holds
\be\label{m1theo}
\forall \alpha \in [0,1], \qquad G(\alpha \delta_0) = g_b(Z,\alpha)
\ee
where $g_b(Z,\alpha)$ is given by \eqref{gbgeneral}. Moreover, it holds
\be\label{m1theo-}
G^-(\alpha\delta_0+\rhop)\ge g_b(Z,\alpha)
\ee
for all $\alpha\in[0,1]$ and for all $\rhop$ such that $\rhop(\{0\})=0$ and $\alpha\delta_0+\rhop\in\Pcal^-$.
\end{theo}

\begin{rema}\label{gb-coherent}
In view of the above result, it follows from Theorem \ref{explform} applied to $\rho=\alpha\delta_0$ that the two definitions of $g_b$ in \eqref{gb} and \eqref{gbgeneral} coincide in the case $N=2$. Unfortunately, at the moment for $N\ge3$ we do not have an explicit definition for $g_b$ that would involve a kind of general partial transport as in \eqref{gb}.
\end{rema}

\begin{rema}\label{solM1}
It follows from \eqref{m1theo} and \eqref{m1theo-} that, in the case $M=1$, the minimum of the $\Gamma$-limit $G$ is attained for any $\rho = \alpha \delta_0$ such that $\alpha$ minimizes $g_b(Z,\cdot)$ on $[0,1]$ : since this function is non-increasing, we note that in particular the probability $\delta_0$ is thus a particular solution. It would be the unique solution in case $g_b(Z,\alpha)$ attains has unique minimum $\alpha=1$ on $[0,1]$, which seems a reasonable conjecture but still an open question.
\end{rema}

\begin{proof}[Proof of Theorem \ref{P2M1}]
We first show \eqref{m1theo-}. Take admissible $\alpha$ and $\rhop$ and consider
a family $\{\rho_\eps\}_{\eps>0}$ in $\PP$ weakly* converging to $\alpha \delta_0+\rhop$. Fix $\eta>0$, then for $r>0$ small enough we have $\int_{B(0,r)} \rhop \leq \eta/2$, so that for $\eps>0$ small enough it holds $\int_{B(0,r)}d\rho_\eps\le\alpha+\eta$. Thus denoting by $\rho_\eps^\eps (x)$ the rescaled version of $\rho_\eps$ 
$$\rho_\eps^\eps(x)=\eps^3\rho_\eps(\eps x)$$
we have 
$$\int_{B(0,r/\eps)}d\rho_\eps^\eps\leq\alpha+\eta\qquad\mbox{and}\qquad G_\eps(\rho_\eps)=G_1(\rho_\eps^\eps)\ge g_b^{r/\eps}(Z,\alpha+\eta)$$
where $g_b^R(Z,\cdot)$ is extended on $\R_+$ as in the proof of Lemma \ref{gblemm}.
This yields $\liminf_{\eps \to 0} G_\eps (\rho_\eps) \ge g_b(Z,\alpha+\eta)$ and since this holds for any such family $\{\rho_\eps\}_\eps$ we infer $G^-(\alpha \delta_0+\rhop) \geq g_b(Z,\alpha+\eta)$. The claim then follows by continuity of $\alpha \mapsto g_b(Z,\alpha)$ on $[0,1]$.

We now turn to \eqref{m1theo}, so we have to prove the inequalities
\[
\forall \alpha \in [0,1], \qquad g_b(Z,\alpha) \leq G^-(\alpha \delta_0) \leq G^+(\alpha \delta_0) \leq g_b(Z,\alpha)\,.
\]
Note that $G^-(\alpha \delta_0) \geq g_b(Z,\alpha)$ follows from the preceding with $\rhop = 0$.
It remains to show $G^+(\alpha \delta_0) \leq g_b(Z,\alpha)$. We first note that this holds for $\alpha=1$: indeed, in that case one has
\[
g_b(Z,1) \;=\; \inf\{ G_1(\rho) : \rho \in \PP\} \;\geq\; G^+(\delta_0)
\]
where in the last inequality we again use that $G_1(\rho) = G_\eps(\rho^{1/\eps})$ for all $\eps >0$ and $\rho^{1/\eps} \weak \delta_0$ as $\eps \to 0$.
Let now $\alpha$ such that $0\le\alpha<1$ and consider a family $\{\rho_R\}_{R>0}$ in $\PP$ such that
$$\int_{B(0,R)}d\rho_R\le\alpha\quad\forall R>0
\qquad\ \text{and}\ \qquad
\lim_{R \to +\infty} G_1(\rho_R) = g_b(Z,\alpha)\,.$$
Up to extracting a subfamily, we may assume that $\{\rho_R\}_{R >0}$ weakly* converges to some $\rho \in \PP^-$. Then one has $\int d\rho = \beta$ for some $\beta \leq \alpha$. We infer :
\begin{itemize}
\item for fixed $\eps>0$, $\rho_R^{1/\eps} \weak \rho^{1/\eps}$ as $R \to +\infty$,
\item $\rho^{1/\eps} \weak \beta\,\delta_0$ as $\eps \to 0$,
\end{itemize}
and since the weak* topology on $\PP^-$ is metrizable, we can extract a subfamily $\{\rho^{}_{R_\eps}\}_{\eps >0}$
such that $\rho^{1/\eps}_{R_\eps} \weak \beta\,\delta_0$ with $R_\eps \to +\infty$. Now we compute
\[
\limsup_{\eps >0} G_\eps\left(\rho^{1/\eps}_{R_\eps}\right) = \limsup_{\eps >0} G_1 \left(\rho^{}_{R_\eps}\right) = g_b(Z,\alpha)\,.
\]
Since $\rho^\eps_{R_\eps} \weak \beta\,\delta_0$, this implies $G^+(\beta \delta_0) \leq g_b(Z,\alpha)$.
If $\beta=\alpha$, the proof is complete. Otherwise $\beta<\alpha$ and thanks to the first step we infer
\[
\qquad g_b(Z,\beta) \leq G^-(\beta \delta_0) \leq G^+(\beta \delta_0) \leq g_b(Z,\alpha)\,.
\]
Since $g_b(Z,\cdot)$ is convex non-increasing, this implies that this function is constant on $[\beta,1]$. But then we have
\[
G^+(\beta \delta_0) \leq g_b(Z,\alpha) = g_b(Z,1) \qquad \text{and} \qquad G^+(\delta_0) \leq g_b(Z,1) = g_b(Z,\alpha)
\]
and by convexity of $G^+$ on $[\beta\,\delta_0,\delta_0]$ we get the desired inequality $G^+(\alpha \delta_0) \leq g_b(Z,\alpha)$.
\end{proof}

We conclude this subsection with some properties of the function $g_b(Z,\cdot)$.

\begin{prop}\label{gbprop}
It holds 
\[
\forall \alpha \in \left[0,\frac{1}{N}\right], \qquad g_b(Z,\alpha) = -\frac{Z^2}{4}\alpha
\]
and
\[
g_b(Z,\alpha) > -\frac{Z^2}{4}\alpha \quad \text{whenever $\Cb(\rho)>0$ for any $\rho \in \Pcal^-$ such that $\int d\rho = \alpha$}.
\]
\end{prop}

\begin{rema}\label{gbpropN2}
In the case $N=2$, it follows from Remark \ref{CbN2} that $\Cb(\rho)>0$ for any $\rho \in \Pcal^-$ such that $\int \rho > \frac{1}{2}$, so in that case $g_b(Z,\alpha) > -\frac{Z^2}{4}\alpha$ for any $\alpha > \frac{1}{2}$.
\end{rema}

\begin{proof}[Proof of Proposition \ref{gbprop}]
It follows from Theorems \ref{alphaileq1N} and \ref{P2M1} applied to $\rho = \alpha \delta_{X_1}$ that $g_b(Z,\alpha) = -\frac{Z^2}{4}\alpha$ whenever $\alpha \leq \frac{1}{N}$.\\
We now turn to the second claim, and assume that $\Cb(\rho)>0$ for any $\rho \in \Pcal^-$ such that $\int d\rho = \alpha$. From Lemma \ref{gblemm} we already know that $g_b(Z,\alpha) \geq -\frac{Z^2}{4}\alpha$, we assume by contradiction that
$g_b(Z,\alpha) = -\frac{Z^2}{4}\alpha$. Then there exists a sequence $(\rho_n)_n$ in $\Pcal$ and a sequence $R_n \to +\infty$ such that
\[
T(\rho_n) + b C(\rho_n) - Z\,U_0(\rho_n) \to -\frac{Z^2}{4}\alpha \quad \text{and} \quad
\forall n, \; \int_{B(0,R_n)} d\rho_n \leq \alpha\,.
\]
If we set $u_n := \sqrt{\rho_n}$ then $\int u^2_n\,dx=1$ for all $n$ and
since $C(\rho_n) \geq 0$ we get 
\[
\limsup_n\int\left[|\nabla u_n|^2-Z\,\frac{u_n^2}{|x|}\right]\,dx = \limsup \left[ T(\rho_n) - Z\,U_0(\rho_n) \right] \leq -\frac{Z^2}{4}\alpha\,.
\]
Now up to extracting a subsequence we may assume that $(u_n)$ weakly converges in $L^2(\R^3)$ to some function $u$ and the above limsup is a limit, so we can apply Lemma \ref{strong} below and get
\[
\int u^2 dx \geq \alpha \quad \text{and} \quad
\int \left[|\nabla u|^2-Z\,\frac{u^2}{|x|}\right]\,dx \le\lim_n\int\left[|\nabla u_n|^2-Z\,\frac{u_n^2}{|x|}\right]\,dx \leq -\frac{Z^2}{4}\alpha \,.
\]
From the properties of $\rho_n$ we infer that $\int_{B(0,R_n)} u_n^2 \leq \alpha$ for all $n$, so that $\int u^2 dx \leq \alpha$. This implies that $\int u^2 dx = \alpha$, and
from \eqref{usualtrick} we get
\[
\int \left[|\nabla u|^2-Z\,\frac{u^2}{|x|}\right]\,dx \geq -\frac{Z^2}{4}\alpha \,.
\]
Summarizing we obtain
\[
\lim_n \left[ T(\rho_n) - Z\,U_0(\rho_n) \right] =
\lim_n\int\left[|\nabla u_n|^2-Z\,\frac{u_n^2}{|x|}\right]\,dx = -\frac{Z^2}{4}\alpha
\]
and then $\lim C(\rho_n)=0$. On the other hand, $\liminf C(\rho_n) \geq \Cb(u^2) >0$ since
$\int u^2 = \alpha$, which is the desired contradiction.
	\end{proof}
\begin{lemm}\label{strong}
Let $(u_n)$ be a sequence in $H^1(\R^3)$ that weakly converges in $L^2(\R^3)$
to some function $u$, and such that
$$\forall n, \quad \int u^2_n\,dx=1,\qquad \text{and} \qquad \lim_n\int\left[|\nabla u_n|^2-Z\,\frac{u_n^2}{|x|}\right]\,dx \leq -\frac{Z^2}{4}\alpha$$
for some $\alpha \in \,]0,1]$. Then $(u_n)_n$ weakly converges in $H^1(\R^3)$ to $u$ and
\[
\int u^2 dx \geq \alpha \quad \text{and} \quad
\int \left[|\nabla u|^2-Z\,\frac{u^2}{|x|}\right]\,dx \leq
 \lim_n\int\left[|\nabla u_n|^2-Z\,\frac{u_n^2}{|x|}\right]\,dx \leq -\frac{Z^2}{4}\alpha \,.
\]
\end{lemm}

Note that in the above result when $\alpha=1$ it follows that $(u_n)_n$ strongly converges in $L^2(\R^3)$ to $u$ and the function $u^2$ is a solution of the problem \eqref{usualtrick}.

\begin{proof}
We first note that for $n$ large enough one has
\[
\int |\nabla u_n|^2 \leq Z\, \int \frac{u_n^2}{|x|}\,dx 
\]
which together with Lemma \ref{l1} yields that $(\nabla u_n)_n$ is bounded in $L^2(\R^3,\R^3)$. Since $(u_n)_n$ is also bounded in $L^2(\R^3)$, we infer that
it is bounded in $H^1(\R^3)$, so it converges weakly in $H^1(\R^3)$ to $u$.
By the weak lower semicontinuity of the $H^1$ seminorm we obtain
$$\int|\nabla u|^2\,dx\le\liminf_n\int|\nabla u_n|^2\,dx.$$
Moreover, for every $R>0$ we have
$$\limsup_n\int\frac{u_n^2}{|x|}\,dx\le\limsup_n\bigg[\int_{B_R}\frac{u_n^2}{|x|}\,dx+\frac1R\bigg]=\int_{B_R}\frac{u^2}{|x|}\,dx+\frac1R$$
where $B_R$ denotes the ball in $\R^3$ of radius $R$ centered at the origin. Since $R$ is arbitrary, we get
$$\limsup_n\int\frac{u_n^2}{|x|}\,dx\le\int\frac{u^2}{|x|}\,dx.$$
Then we deduce
\be\label{uineq}
\int\Big[|\nabla u|^2- Z \frac{u^2}{|x|}\Big]\,dx\le\lim_n\int\Big[|\nabla u_n|^2-Z \frac{u_n^2}{|x|}\Big]\,dx \leq -\frac{Z^2}{4}\alpha\,.
\ee
It remains to prove that $\int u^2\,dx=\beta \geq \alpha$. Assume by contradiction that $\beta < \alpha$. We first note that $\beta >0$ otherwise $u=0$ which contradicts the inequalities in \eqref{uineq}. Then the probability $\rho = \frac{1}{\beta} u^2$ satisfies
$$\int\left[\frac{|\nabla \rho|^2}{4\rho}-Z \frac{\rho}{|x|}\right]\,dx \le -\frac{Z^2}{4} \frac{\alpha}{\beta}$$
which contradicts \eqref{usualtrick} since $\frac{\alpha}{\beta} > 1$.
\end{proof}

\subsection{Perspectives and future work}
A general proof (for any $N$ and $M$) of \ref{P1} and \ref{P2} would give a full characterization of the $\Gamma$-limit functional $G$. It would then be even more interesting if the function $g_b$ introduced in Definition \ref{defgb} could be interpreted as the ground state energy of a molecule with one nucleus.
This is precisely what we obtained in Section \ref{two-interact} in the case $N=2$.
It seems that a necessary tool for this program is an expression for the relaxation $\Cb$
of the transport cost $C$ with respect to the weak* convergence of measures. It would be also interesting to carry out a study (numerical and theoretical) of the minimizers of the $\Gamma$-limit functional $G$ which could explain how the bond dissociation happens (i.e. how the electrons are divided among the resulting molecules). A numerical study could also help to understand if the function $\alpha \mapsto g_b(Z,\alpha)$, which is non increasing in $[0,1]$, attains its minimum uniquely for $\alpha=1$ (see Remark \ref{solM1}).

Another interesting issue is the existence, for a fixed $\eps>0$, of minimizers $\rho_\eps\in\PP$ for the functional $F_\eps$ defined in \eqref{Feps}. The existence of a solution $\bar\rho_\eps\in\PP^-$ for the relaxed functional $\bar F_\eps$ (with respect to the weak* convergence of measures) follows straightforwardly from the direct methods of the calculus of variations; the question if $\int d\bar\rho_\eps=1$ is sometimes called {\it ionization conjecture} and is part of our future work, together with a complete characterization of the relaxed correlation functional $\overline C$.

\section*{Acknowledgements}

This paper has been written during some visits of the authors at the Department of Mathematics of University of Firenze, University of Pisa, and at the Laboratoire IMATH of University of Toulon. The authors gratefully acknowledge the warm hospitality of these institutions.

During the work, several conversations with Simone Di Marino, Paola Gori-Giorgi, Mathieu Lewin, and Michael Seidl helped to understand the physical problem and to better frame it in the literature; they are all warmly acknowledged.

The research of the second and fourth authors is part of the project 2010A2TFX2 {\it Calcolo delle Variazioni} funded by the Italian Ministry of Research and is partially financed by the {\it``Fondi di ricerca di ateneo''} of the University of Firenze.

%%%%%%%%%%%%%%%%%%%%%%%%%%%%%%%%%%%%%%%%%%%%%%%%%%

\bigskip
{\small\noindent
Guy Bouchitt\'e:
Laboratoire IMATH, Universit\'e de Toulon\\
BP 20132, 83957 La Garde Cedex - FRANCE\\
{\tt bouchitte@univ-tln.fr}\\
{\tt https://sites.google.com/site/gbouchitte/home}

\bigskip\noindent
Giuseppe Buttazzo:
Dipartimento di Matematica, Universit\`a di Pisa\\
Largo B. Pontecorvo 5, 56127 Pisa - ITALY\\
{\tt giuseppe.buttazzo@unipi.it}\\
{\tt http://www.dm.unipi.it/pages/buttazzo/}

\bigskip\noindent
Thierry Champion:
Laboratoire IMATH, Universit\'e de Toulon\\
BP 20132, 83957 La Garde Cedex - FRANCE\\
{\tt champion@univ-tln.fr}\\
{\tt http://champion.univ-tln.fr}

\bigskip\noindent
Luigi De Pascale:
Dipartimento di Matematica e Informatica, Universit\`a di Firenze\\
Viale Morgagni 67/a, 50134 Firenze - ITALY\\
{\tt luigi.depascale@unifi.it}\\
{\tt http://web.math.unifi.it/users/depascal/}

\end{document}